\documentclass{amsart}

\usepackage{amsmath}
\usepackage{amssymb}
\usepackage{bbm}
\usepackage{upgreek}
\usepackage{marvosym}
\input{quivers.sty}
\input{polytope.sty}

\newtheorem{theorem}{Theorem}[section]
\newtheorem{lemma}[theorem]{Lemma}

\newtheorem{corollary}[theorem]{Corollary}
\newtheorem{proposition}[theorem]{Proposition}

\theoremstyle{definition}
\newtheorem{definition}[theorem]{Definition}
\newtheorem{example}[theorem]{Example}

\theoremstyle{remark}
\newtheorem{remark}[theorem]{Remark}

\newtheorem{problem}[theorem]{Problem}

\numberwithin{equation}{section}

\newcommand{\GL}{\operatorname{GL}}

\newcommand{\op}{\operatorname{op}}
\newcommand{\Vect}{\operatorname{Vect}}
\newcommand{\Rep}{\operatorname{Rep}}
\newcommand{\Mod}{\operatorname{Mod}}
\newcommand{\Fun}{\operatorname{Fun}}
\newcommand{\Hom}{\operatorname{Hom}}
\newcommand{\Inj}{\operatorname{Inj}}

\newcommand{\Aut}{\operatorname{Aut}}
\newcommand{\Pic}{\operatorname{Pic}}
\newcommand{\Inn}{\operatorname{Inn}}
\newcommand{\End}{\operatorname{End}}
\newcommand{\Ext}{\operatorname{Ext}}
\newcommand{\ext}{\operatorname{ext}}
\newcommand{\Ker}{\operatorname{Ker}}
\newcommand{\Coker}{\operatorname{Coker}}

\newcommand{\Id}{\operatorname{Id}}

\newcommand{\alg}{\operatorname{alg}}

\newcommand{\D}[1]{\mathcal{#1}}
\newcommand{\B}[1]{\mathbb{#1}}
\renewcommand{\-}{$-$}
\newcommand{\br}[1]{\overline{#1}}
\newcommand{\Br}{\operatorname{Br}}
\newcommand{\mr}[1]{\mathring{#1}}
\newcommand{\innerprod}[1]{\langle#1\rangle}
\newcommand{\sm}[1]{\left(\begin{smallmatrix}#1\end{smallmatrix}\right)}

\makeatletter
\def\equalsfill{$\m@th\mathord=\mkern-7mu
\cleaders\hbox{$\!\mathord=\!$}\hfill
\mkern-7mu\mathord=$}
\makeatother

\begin{document}

\title{Categorical Homotopy I. Quivers}

\author{Jiarui Fei}
\address{MSRI, 17 Gauss way, Berkeley, CA 94720, USA}
\email{jfei@msri.org}
\thanks{}


\date{}
\dedicatory{}

\begin{abstract} We quiver-interpret the classical simplicial theory - including the cosimplex category $\Delta$, Dold-Kan correspondence, and Hochschild homology - as a certain Q-homotopy theory of type $A$.
For the cyclic and cubical theories, we proceed analogously. Subsequently, we present far-reaching generalizations, using different types of quivers. Moreover, we explain how to construct certain categories as analogs of $\Delta$, and associate to each a Q-homotopy theory. We provide many examples, including such theories of type $D$.
\end{abstract}

\maketitle

\section*{Introduction}

In this series of notes, we try to develop a new homological algebra rooted in category and representation theory, especially quiver theory, rather than in classical topology. We hope that it can refine, if not replace, the traditional homological algebra in the future.

The following construction lies at the heart of the algebraic topology and homological algebra.
Consider a functor $S_{\D{C}}:S\to \D{C}$, where $\D{C}$ is a cocomplete category possibly $V$-enriched for some ``nice" category $V$.

It induces a pair of adjoint functors
$$|-|:\Fun(S^{\op},V)\rightleftharpoons \D{C}:N$$
between $\D{C}$ and the functor category $\Fun(S^{\op},V)$,
where $N$ behaves like a {\em nerve operation} and $|-|$ behaves like {\em geometric realization}.
In such a situation, one is led to study objects $c\in\D{C}$ in terms of their nerves $N(c)$ through the machinery of homological algebra. The situation is particularly nice if $S_{\D{C}}$ is a {\em dense functor}, as this ensures that the corresponding
nerve functor is fully faithful.
As a classical example, we take $S$ to be the cosimplex category $\Delta$, $\D{C}$ the category of topological spaces, with $S_{\D{C}}$ sending the abstract $n$-simplex $[n]$ to the standard topological $n$-simplex. In this case, $|-|$ is the usual geometric realization and $N$ the singular simplicial complex functor. Almost all classical cohomology theories arise from this construction for various choices of $\D{C}$. On the other hand, there have been only a few attempts to vary $S$.
In the 1980's, Grothendieck in {\em \`{A} la poursuite des champs} introduced {\em test categories} as certain variants of $\Delta$. The {\em weak equivalence} involved in his original definition still refers to the simplicial setting, but later he conjectured a more intrinsic characterization of this simplicial notion of weak equivalence, which was proved by Cisinski in \cite{Ci}. By contrast, we will radically depart from the simplicial world.

From a view of {\em quivers}, the category $\Delta^{\op}$ can be realized as follows:
the objects $[n]$ in $\Delta^{\op}$ are the module categories of quivers $A_{n+1}:$
$$\An{0}{}{1}{n-1}{}{n}{}{}.$$
the $i$-th coface map $\pi_i: [n]=\Mod(A_{n+1})\to [n-1]=\Mod(A_{n})$ is the right {\em orthogonal projection} from the simple representation $S_{i}$ for $i=0,1,\dots,n$;
the $i$-th degeneration map $\iota_i: [n-1]\to[n]$ is the exact embedding right adjoint to $\pi_i$.
\begin{equation}\label{eq:simpq} \vcenter{\xymatrix@C=5ex@R=0ex{
[n]  \ar@{_<-^>}[rr]^{\pi_n,\dots,\pi_0}_{\iota_{n-1},\dots,\iota_{0}} &&
[n-1]  \ar@{_<-^>}[r] & \cdots \ar@{_<-^>}[r] & [2]  \ar@{_<-^>}[rr]^{\pi_2,\pi_1,\pi_0}_{\iota_1,\iota_0} &&
[1] \ar@{_<-^>}[rr]^{\pi_1,\pi_0}_{\iota_0} && [0] }
}\end{equation}
One can verify the {\em simplicial identities}:
\begin{align*} & \pi_j \pi_i = \pi_{i-1} \pi_j\quad \text{ if } j<i, \\
& \pi_j \iota_i= \begin{cases} \iota_{i-1} \pi_j & \text{if } j<i, \\
e_* & \text{if } j=i \text{ or } i+1, \\
\iota_{i} \pi_{j-1} & \text{if } j>i+1, \end{cases} \\
& \iota_i \iota_j = \iota_{j} \iota_{i-1}\qquad \text{ if } j<i. \end{align*}

There is nothing too special about the $A_n$-quivers in the quiver world. So the first purpose of this note is to consider other choices of $S$ from the quiver setting. We will define the category $\Delta(\B{Q})$ for every {\em coface configuration} $\B{Q}$.

A {\em simplicial object} in an abelian category $\D{A}$ is a functor $\Delta^{\op}\to\D{A}$. In our language, it is a {\em representation} of the {\em quiver with relations} \eqref{eq:simpq} in $\D{A}$.
To every simplicial object, one can associate chain complexes, which are representations of the linear quiver with relations $dd=0$:
$$\An{n}{d}{n-1}{2}{d}{1}{d}{d}.$$
Historically the first such a complex is the {\em unnormalized Moore complex}. Later there is the {\em normalized} construction and {\em Dold-Kan correspondence}, which is the key bridge connecting the classical homological algebra and homotopy theory. Let $N$ be the normalized chain complex functor from the category of simplicial objects to the category of chain complexes. The Dold-Kan correspondence says that $N$ has a left adjoint $K$ such that both $KN$ and $NK$ are identities.
This fact is related to {\em Morita equivalence} of the algebras of the above two quivers with relations.
The second purpose of us is to define a {\em quiver complex} associated to every functor $\Delta(\B{Q})\to\D{A}$.
This is done by generalizing the Dold-Kan correspondence through Morita theory. 

The word quiver in the title has at least three-fold meaning. First we use quivers to construct the category $\Delta(\B{Q})$. Second we associate a quiver complex to any functor $\Delta(\B{Q})\to\D{A}$. Finally, we use quivers to construct concrete {\em Q-homotopy theories}, generalizing most classical (co)homology theories, such as \v{C}ech, Koszul, de Rham, Hochschild, Connes's cyclic homology. Let us take the Hochschild complex as an example. The $i$-th coface map $A^{\otimes n}\to A^{\otimes n-1}$ can be visualized as:
\begin{align*}
\An{0}{a_1}{1}{n-1}{a_{n}}{n}{a_2}{a_{n-1}}\qquad\quad \mapsto\\
\xymatrix@C=5ex{0 \ar[r]^{a_1}  & \cdots \ar[r]^{a_{i-1}} & i-1 \ar[rr]^{\quad a_ia_{i+1}} & & i \ar[r]^{a_{i+2}\ } & \cdots \ar[r]^-{a_n} & n-1}
\end{align*}
resembling the application of the functor $\pi_i$ on a representation of dimension $(1,1,\dots,1)$.
We summarize in a slogan that the {\em classical simplicial theory is a Q-homotopy theory of type $A$}.

This notes is organized as follows.
In Section \ref{S:QRS}, we introduce the notion {\em quiver with relations and sections (QRS)}. It provides a very general framework for us to work with throughout. We study Morita theory for the algebra of a QRS. Our first main result, Theorem \ref{T:Br}, offers a way to find the associated basic algebra with the equivalence functor, using certain elementary operation called {\em breaking}. Then we introduce in Definition \ref{D:QHT} the Q-homotopy theory in the category of $k$-modules. We work out several simple examples. In the end, we introduce the {\em GReedy condition} which will be considered in Section \ref{S:Delta}.

In Section \ref{S:Qcomp}, we generalize the definition of Q-homotopy theory from the category of $k$-modules to any abelian category. The key observation is Lemma \ref{L:Meq}. Then we explain why this definition generalizes the classical theories by first quiver-interpreting the classical Dold-Kan correspondence (Proposition \ref{P:CDK}). The results on {\em symclic} and cubical objects (Proposition \ref{P:SDK} and \ref{P:CuDK}) seem mildly new.

In Section \ref{S:OP}, we review some basics on the orthogonal projections in the category of quiver representations. Those are the building blocks for the {\em quiver exceptional configurations (QEC)} in the next section. Our second main results are Theorem \ref{T:rel} and Corollary \ref{C:rel2}, which are extremely useful for finding relations in a QEC. Other results like Lemma \ref{L:sproj} and its corollary will be frequently used as well.

In Section \ref{S:Delta}, we introduce in Definition \ref{D:DeltaQ} our most important category $\Delta(\B{Q})$ for any coface configuration $\B{Q}$. We are mainly interested in the coface configurations in the category $\D{CC}(Q)$ (Definition \ref{D:CCQ}). Using this, we quiver-interpret many classical categories, such as the simplex category, the cycle category, and the $n$-cube category (Example \ref{Ex:BAn}, \ref{Ex:BAncyc}, and \ref{Ex:BSn}). Based on many examples, we propose in Definition \ref{D:QECgen} and \ref{D:MGC} a way to produce interesting QEC's.


In Section \ref{S:QECwd}, we introduce a subclass of QEC's called QEC with dimension vectors (QECwd), which is crucial for our application in the last section. We have two interesting examples, which can be called QECwd's of type $D$ in certain sense. Their corresponding Q-homotopy theories are studied in Proposition \ref{P:Dn1} and \ref{P:Dn2}. These are our third main results. Example \ref{Ex:BBn} is another interesting example.

In Section \ref{S:deloop}, we introduce {\em delooped} coface configuration. This makes the computation of the {\em classical} Q-homotopy much easier. We quiver-interpret the symcle category (Example \ref{Ex:symcle}), and give examples of delooped QECwd's of type $D$ (Example \ref{Ex:Dn1dl} and \ref{Ex:Dn2dl}).

In Section \ref{S:App}, we explain how to start from a QECwd and construct concrete Q-homotopy theory for general mathematical objects. We do this especially for algebras, generalizing Hochschild and cyclic homology. Together with previous examples, we fully construct Q-homotopy theory of type $D$ for commutative algebras in Example \ref{Ex:NAsrooted} and \ref{Ex:NADn2}. 

In Section \ref{S:Discuss}, we propose several important open problems. In Appendix, we mention the unnormalized construction in our setting. After we finished this work, we feel that this construction is not as essential as the normalized one, but it is still quite interesting, especially for computational purpose. So we put this part in the appendix.

We always believe that the great Tao must be extremely simple.
We hope that this notes is accessible to most graduate students in mathematics. Some exposure to homological algebra and quiver theory is enough. Maybe the only involved section is Section \ref{S:OP}. However, this section is not crucial to understand our main ideas, especially if one is willing to accept results there.

\subsection*{Definitions, Notations, and Conventions}

We use \cite{We} as our standard reference on homological algebra; \cite{ASS} on representation theory of basic algebras; and \cite{DW2} especially on quiver theory. Readers should be able to find most unexplained definitions and notations there.

All our vectors are assumed to be row vectors, but all modules are left modules unless otherwise stated. This is opposite to the usual convention in representation theory that row vectors go with right modules, or vice versa. Unadorned $\otimes$'s and $\Hom$'s are always over the base ring $k$. The trivial dual $^*$ is $\Hom(-,k)$.

A quiver $Q$ is a directed graph with multiple arrows defined by a quadruple $(Q_V,Q_A,t,h)$. $Q_V$ and $Q_A$ are the sets of vertices and arrows respectively, and both $t$ and $h$ are maps from $Q_A$ to $Q_V$. An arrow $a\in Q_A$ goes from its {\em tail} $ta$ to its {\em head} $ha$. A quiver $Q$ is {\em finite} if both $Q_V$ and $Q_A$ are finite sets. We simplify the multiple arrows:
$$\vcenter{\xymatrix@C=9ex{
\bullet \ar@<2ex>[r]|{a} \ar@<1ex>[r]|{b} \ar@<0ex>[r]|{c} & \bullet \ar@<1ex>[l]|{d} \ar@<2ex>[l]|{e} &  \text{by}  &
\bullet \ar@{_<-^>}[r]^{a,b,c}_{d,e} & \bullet }
}$$

A path in $Q$ is a sequence of arrows $p = a_1a_2\cdots a_s$, with $ta_i=ha_{i+1}$ for all $i$. We define $tp=ta_s$ and $hp=ha_1$. For each vertex $v\in Q_V$, we also define the trivial path $e_v$ of length 0, satisfying $te_v=he_v=v$. An oriented cycle is a nontrivial path satisfying $hp=tp$. All paths in $Q$ form a semigroup, where the product of two paths $p$ and $q$ is by definition their concatenation if $tp=hq$; and zero otherwise.

We fix a PID $k$. The {\em path algebra} $kQ$ is the semigroup ring over $k$. The path algebra is bigraded: $kQ=\bigoplus_{u,v\in Q_V}e_vkQe_u$. A {\em relation} $r=\sum_{i=1}^s c_ip_i$ with $c_i\in k$ and $p_i$ a path, is always assumed to be homogeneous with respect to the grading, i.e., there exist $tr,hr\in Q_V$ such that $tp_i=tr$ and $hp_i=hr$ for all $i$. Let $R$ be a set of relations. The algebra $k\D{Q}$ of the {\em quiver with relations} $\D{Q}=(Q,R)$ is formed from the path algebra $kQ$ by quotienting out the ideal $I$ generated by all relations in $R$. We also call $kQ/I$ a {\em quiver presentation} of this algebra $A$.

Let $\Mod(A)$ be the category of finitely generated left $A$-modules.
Let $P_v$ (resp. $I_v$) be the projective (resp. injective) module $Ae_v$ (resp. $e_vA$) corresponding to the vertex $v$. They are characterized by the property that
$$\Hom_A(P_v,N)=Ne_v=\Hom_A(N,I_v)^* \text{ for any } N\in\Mod(A).$$
Let $S_v$ be the rank-one $k$-free module on $v$. When $S_v$ (resp. $P_v$) is a subscript, we simply write $v$ (resp. $-v$). For example, $\pi_v$ is the same as $\pi_{S_v}$. When subscripts are integers, we indulge in the following interval notation: $$\B{S}_{[i,j]}:=\{S_i,S_{i+1},\dots,S_{j}\}.$$
If $i=1$, then we write $\B{S}_j$ instead. We also indulge in the following matrix notation:
$$\oplus \B{P}M:=\bigoplus_{u,v} M_{uv}P_v,$$
where $\B{P}$ is a row vector with entries $P_v$ and $M$ is a positive integer matrix.


\subsection*{A List of Quivers}
We list all quivers that we will frequently use below.
$$A_n: \An{1}{}{2}{n-1}{}{n}{}{}$$
$$\mr{A}_n: \Ancyc{}{}{}{n-1}{n}$$
$$A_n^1: \Anonefixed$$
$$D_n: \Dnfixed$$
$$D_n^*: \Dndual{}{}{}{}{}{n-2}{n-1}{\quad n}$$
$$S_n: \Sn$$
$$B^3:\extendedtameAfour{}{}{}{}{}{},\qquad B^2:\extendedtameAthree{1}{2}{3}{4}{}{}{}{}{},\qquad B^1:\extendedtameAtwo{1}{2}{3}{}{}{u}{l},$$
$$B':\Bprime{1}{2}{3}{}{}{}{},\qquad B^0:\Bzero{1}{2}{}{}{}{}.$$

\section{Quivers with Relations and Sections} \label{S:QRS}

We fix a PID $k$. Before section \ref{S:OP}, we never think $k$ as any nice field. We suggest that readers take $k$ to be the ring of integers $\B{Z}$.
Let $Q$ be a finite quiver and $R$ be a set of relations of $Q$. If two non-loop arrows $a$ and $a^r$ satisfy relation $aa^r=e_{ha}$, then we call $a^r$ a {\em section} of $a$. We denote the set of all sections by $S$ or $Q_A^r$, and the set of all such relations by $R(S)$. Moreover, we write $Q_A^l$ or $P$ for the set of all {\em projections}, that is, non-loop arrows having a section. We require that $S\cap P$ is empty.

\begin{definition} The above collection of data $\D{Q}=(Q,R,S)$ is called a {\em quiver with relations and sections}, or QRS in short.
\end{definition}

There is an obvious dual way to describe such data using projections.
The corresponding dual notion is called a {\em quiver with relations and projections}, or QRP in short.
We have the $k$-algebra $k\D{Q}$ formed from the path algebra $kQ$ by quotient out the ideal generated by all relations in $R$. We assume throughout that $k\D{Q}$ is finite-dimensional if we do a base change to any field. An easy observation is

\begin{lemma} \label{L:22} The $k$-submodule spanned by $a,a^r,a^ra,e_{ha}$ is isomorphic to the $2\times 2$ matrix algebra over $k$.
\end{lemma}

We must point out that being a section in an algebra $A$ is a relative notion, which depends on the choice of {\em vertex idempotents} in the quiver presentation $kQ/I$ of $A$. The definition of the section and projection has an obvious generalization from arrows to any $a\in A$.
There is an absolute notion called {\em weak section}. If $a,a^r\in A$ such that $aa^r$ is an idempotent, then we call $a^r$ a weak section of $a$, and denote the idempotent $aa^r$ by $e_{ha}$. Despite of the notation $a^r$, we should keep in mind that a fixed $a\in A$ may have more than one (weak) section. The notation $a^r$ is useful for simplifying many things.

\begin{definition} Let $e$ be the unit of $A$ and $e'=e-a^ra$. ``{\em Breaking $a$}" is an operation $\Br_a$ to the algebra $A$ defined by $\Br_a(A)=e'Ae'$. For any $c\in A$ we denote $e'ce'$ by $\Br_a(c)$. Sometimes by abuse of notation, we write $\Br_a$ for the functor $e'(-):\Mod(A)\to\Mod(\Br_a(A))$.
\end{definition}

\begin{lemma} \label{L:Br} The new algebra $\Br_a(A)$ is Morita equivalent to $A$. The inverse right adjoint to the breaking $\Br_a$ is the functor $\Hom_{\Br_a(A)}(\Br_a(A)\oplus P_{ha},-)$, where $P_{ha}=e'Ae_{ha}$ is the projective $\Br_a(A)$-module corresponding to the idempotent $e_{ha}$.
\end{lemma}
\begin{proof} It is enough to show that $Ae_{ha}\cong Aa^ra$ as $A$-modules. The isomorphism is given by the right multiplication by $a$ with inverse the right multiplication by $a^r$. By \cite[Theorem I.6.8]{ASS}, the inverse right adjoint to $\Br_a$ is given by $\Hom_{\Br_a(A)}(e'A,-)$, but
$$e'A=e'A(e'+a^ra)=\Br_a(A)\oplus e'Aa^ra\cong \Br_a(A)\oplus e'Aaa^r = \Br_a(A)\oplus P_{ha}.$$
\end{proof}

Let $C$ be a complete set of {\em primitive orthogonal idempotents} \cite[I.4]{ASS} of $A$.
\begin{definition}
A partition $\bigsqcup_i\{e_{ij}\}$ of $C$ is called {\em basic} if $\{e_{i1}\}_i$ is a {\em basic set} of $C$, that is, if we set $b=\sum_i e_{i1}$, then $A^b=bAb$ is the {\em basic algebra} \cite[Definition I.6.3]{ASS} associated to $A$. We call a quiver presentation of $A$ {\em basic} if its vertex-idempotents $e_i=\sum_j e_{ij}$ form a basic partition of $C$.
\end{definition}

\begin{example} The quiver $\xymatrix{2 \ar@{_<-^>}[r]^{a}_{b} & 1  }$ with relations $ab=e_1,ba=e_2$ is not a basic presentation because it is the $2\times 2$ matrix algebra, which is simple. But this is not a QRS because $a$ is both a section and a projection. The quiver $\xymatrix{1 \ar@(ru,rd)^{a} }$ with relations $a^2=e_1$ is not a basic presentation when $2$ is invertible in $k$ because the algebra is $k\times k$ with $(1,0)$ and $(0,1)$ identified with $\frac{1}{2}(e_1-a)$ and $\frac{1}{2}(e_1+a)$.

\end{example}

In this notes, all quiver presentations are arranged to be basic. If $a$ has a section $a^r$, then $a^ra\neq e_{ta}$, i.e., $a^r$ is not a projection. Otherwise $e_{ta}A\cong e_{ha}A$, which is a contradiction. $\Br_a(A)$ naturally inherits vertex-idempotents from $A$ by replacing $e_{ta}$ by $e_{ta}-a^ra$ but keeping everyone else. By Lemma \ref{L:Br}, the new partition by vertices is basic.
Readers may keep in mind that we will mainly deal with
\begin{equation}\text{algebras associated to QRS's, whose vertex-idempotents give a basic partition.} \tag*{\Smiley} \end{equation}

In general, $\Br_a(a_0^r)$ can be a section but not necessary a section of $\Br_a(a_0)$ in $\Br_a(A)$. However, it is quite clear that this is the case if $ta\neq ta_0$.
Assume $\Br_a(a_0^r)$ is a section for $\Br_a(a_0)$, then we write $\Br_{(a_0,a)}:=\Br_{\Br_a(a_0)}\circ \Br_{a}$. More generally, we can recursively define $\Br_{(a_k,\dots,a_2,a_1)}$ tacitly assuming $\Br_{(a_i,\dots,a_1)}(a_{i+1}^r)$ are all sections of $\Br_{(a_i,\dots,a_1)}(a_{i+1})$ in $\Br_{(a_i,\dots,a_1)}(A)$. We call it {\em breaking at the sequence} $(a_k,\dots,a_2,a_1)$.

\begin{definition} An algebra is called {\em weakly crisp} if one can get its basic algebra by consecutive breakings. It is called {\em crisp} if it is weakly crisp and the consecutive breakings can be given as a breaking at a sequence.
\end{definition}

\begin{theorem} \label{T:Br} Let $B=\Br_{(a_k,\dots,a_2,a_1)}(A)$, then $B$ is Morita equivalent to $A$. The inverse right adjoint to $\Br_{(a_k,\dots,a_2,a_1)}$ is the functor $\Hom_{B}(P,-)$, where $P$ is the projective module $\oplus \B{P}E_1E_2\cdot\cdots\cdot E_k $, and $E_i$ is the elementary matrix obtained from the $n\times n$ identity matrix by adding $1$ on the $(ha_{i},ta_{i})$-entry.
Moreover, for any $A$-module $M$, on each vertex $v$ $\Br_{(a_k,\dots,a_2,a_1)}(M)$ is
$$\Coker(\sum_{ta_i=v}a_i^r)= \bigcap_{ta_i=v}\Ker(a_i).$$
\end{theorem}

\begin{proof} Almost everything follows from Lemma \ref{L:Br} through induction.
For the formula of the inverse functor, we use the adjunction formula: $$\Hom_{A}(M\otimes_{A_1} N,-)=\Hom_{A_1}(M,\Hom_{A}(N,-)).$$
For the last statement, we notice that on each vertex $v$, $\Br_{a}(M)$ is $\Coker(a^r)\cong \Ker(a)$ as a $k$-module, and $\Br_a(a_0),\Br_a(a_0^r)$ are just the restriction of $a_0$ and $a_0^r$.
\end{proof}

Let $b$ be the sum of a basic set of $C$. When the algebra $A$ is (weakly) crisp, we always take $b$ to be the one obtained by applying the breakings to the identity $e$ of $A$. The following two definitions do not depend on the choice of $b$.
\begin{definition} \label{D:complex} Let $\nu:M\mapsto bM$ be the Morita equivalence functor, and we call it {\em normalized quiver complex functor} in our context.
\end{definition}

Let $\underline{\Mod}(A^b)$ be the {\em injectively stable category} \cite[IV.2]{ASS} of $\Mod(A^b)$ and $q$ be the quotient functor $\Mod(A^b)\to \underline{\Mod}(A^b)$.

\begin{definition} \label{D:QHT} The composite $\underline{\pi}=q \nu$ is called the {\em Q-homotopy functor} of $A$. For any $T\in\Mod(A^b)$, the {\em $i$-th classical homotopy relative to $T$, or $T$-homotopy} is the functor $\pi_i(T,-):=\Ext_{A^b}^i(T,\nu(-))$.

We can replace the injectively stable category by the {\em projectively stable category}, and define the {\em Q-cohomotopy functor} $\br{\pi}$. The $T$-cohomotopy functor is $\pi^i(-,T):=\Ext_{A^b}^i(\nu(-),T)$.
\end{definition}

\begin{example} \label{Ex:simple3} Consider the following quiver
$$\vcenter{\xymatrix@R=0ex{
3 \ar@{_<-^>}[rr]^{b_3,b_2,b_1}_{b_2^r,b_1^r} & & 2 \ar@{_<-^>}[rr]^{a_2,a_1}_{a_1^r} & & 1 }
} $$
with relations \begin{align*}
& a_j a_i^r=e_1,\quad b_j b_i^r=e_2, & \text{if } & j=i \text{ or } i+1 \\
& a_j b_i = a_{i-1}b_j, & \text{if } & j<i \\
& b_2^r a_1^r = b_1^r a_1^r,\quad
b_1 b_2^r = a_1^r a_1,\ b_3 b_1^r = a_1^r a_2. &
\end{align*}
If we choose the basic set to be $$\{e_{31}=e_3-b_2^rb_2+b_2^rb_1-b_1^rb_1,e_{21}=e_2-a_1^r a_1,e_{11}=e_1\},$$
then $A^b$ is the algebra presented by
$$\vcenter{\xymatrix{
3 \ar[rr]^{b_3'} && 2 \ar[rr]^{a_2'} && 1  }
}$$
where $b_3'=e_{21}b_3e_{31}=b_3-b_2+b_1-b_3b_1^rb_1$ and $a_2'=e_{11}a_2e_{21}=a_2-a_1$. Interested readers can verify that this set of idempotents is obtained from the breaking at the sequence $(b_2,a_1,b_1)$ or $(a_1b_2,b_2,b_1,a_1)$.
Note that $a_2'b_3'=0$ gives the relation of $A^b$.

Moreover on $\Ker(b_1)\cap\Ker(b_2)$, $b_3-b_2+b_1-b_3b_1^rb_1$ reduces to $b_3$, but on $\Coker(b_2^r-b_1^r)$, it reduces to $b_3-b_2+b_1$.
So $\nu(M)$ can be simplified as
\begin{align*}\vcenter{\xymatrix{
\Coker(b_2^r-b_1^r) \ar[rr]^{\quad b_3-b_2+b_1} && \Coker(a_1^r) \ar[rr]^{\ a_2-a_1} && M(1) }
}\\
\vcenter{\xymatrix{
\Ker(b_1)\cap\Ker(b_2) \ar[rr]^{\qquad b_3} && \Ker(a_1) \ar[rr]^{a_2} && M(1) }
}\end{align*}
By Theorem \ref{T:Br}, the inverse of $\nu$ is represented by the projective representation $P=\oplus\sm{1&2&1\\0&1&1\\0&0&1}\sm{P_3\\P_2\\P_1}$. Note that two sequence of breakings give two factorizations:
$$\sm{1&0&1\\0&1&0\\0&0&1}\sm{1&1&0\\0&1&0\\0&0&1}\sm{1&1&0\\0&1&0\\0&0&1}\sm{1&0&0\\0&1&1\\0&0&1}=\sm{1&2&1\\0&1&1\\0&0&1}=\sm{1&1&0\\0&1&0\\0&0&1}\sm{1&0&0\\0&1&1\\0&0&1}\sm{1&1&0\\0&1&0\\0&0&1}.$$
\end{example}

\begin{example} \label{Ex:cycle3} We add some loops to the previous example:
$$\vcenter{\xymatrix@R=0ex{
3 \ar@{_<-^>}[rr]^{b_3,b_2,b_1}_{b_2^r,b_1^r} \ar@(ul,ur)|{t} & & 2 \ar@{_<-^>}[rr]^{a_2,a_1}_{a_1^r} \ar@(ul,ur)|{t} & & 1 }
} $$
satisfying additional relations \begin{alignat*}{2}
t^i & =e_i, & a_i t & =a_j\quad (i\neq j),\\
b_i^{(r)}t & =tb_{i-1}^{(r)}, & b_1t & = b_3, b_1^r t = t^2 b_2^r.
\end{alignat*}
Choose the same vertex-idempotents, then $A^b$ is the algebra presented by
$$\vcenter{\xymatrix{
3 \ar@{_<-^>}[rr]^{b'_3}_{{b'}^{r}} \ar@(ul,ur)|{t'} && 2 \ar@{_<-^>}[rr]^{a'_2}_{{a'}^{r}} \ar@(ul,ur)|{t'} && 1  }
},$$
where ${b'}^{r}=e_{31}(-tb_2^r)e_{21}=-tb_2^r-b_2^r+b_1^r+tb_2^rb_1b_2^r, {a'}^{r}=e_{21}(ta_1^r)e_{11}=ta_1^r-a_1^r$, and $t'=e_{i1}te_{i1}$. We have the following relations for $A^b$:
\begin{align*} & a_2'b_3'=0,\quad {b'}^{r}{a'}^{r}=0, \\
& a_2'{a'}^{r}=0,\quad b_3'({b'}^{r}-{b'}^{r}t')+{a'}^{r}a_2'=0, \\
& \text{and many other relations involving }t'. 
\end{align*}
Note that if $2$ or $3$ is invertible in $k$, the algebra $A^b$ is not basic.
\end{example}

\begin{example} \label{Ex:symcle3} Instead of adding loops, we add arrows $b^r=-tb_2^r(e_3-t),a^r=ta_1^r:$
$$\vcenter{\xymatrix@R=0ex{
3 \ar@{_<-^>}[rr]^{b_3,b_2,b_1}_{b^r,b_2^r,b_1^r} & & 2 \ar@{_<-^>}[rr]^{a_2,a_1}_{a^r,a_1^r} & & 1 }
} $$
Choose the same vertex-idempotents, then $A^b$ is the algebra presented by
$$\vcenter{\xymatrix{
3 \ar@{_<-^>}[rr]^{b_3'}_{B^{r}} && 2 \ar@{_<-^>}[rr]^{a_2'}_{A^{r}} && 1  }
}$$
with relations \begin{align*} & a_2'b_3'=0,\quad B^{r}A^{r}=0, \\
& a_2'A^{r}=0,\quad b_3'B^{r}+A^{r}a_2'=0,
\end{align*}
where $B^r=b^r+b_2^rb_2tb_2^r-b_2^rtb_3b_1^r-b_2^r+b_2^rt+b_1^r-b_1^rt, A^r=a^r-a_1^r$.
This time $\nu(M)$ can be simplified as: $$\vcenter{\xymatrix{
\Coker(b_2^r-b_1^r) \ar@{_<-^>}[rr]^{\quad b_3-b_2+b_1}_{\quad b^r} && \Coker(a_1^r) \ar@{_<-^>}[rr]^{\ a_2-a_1}_{\ a^r} && M(1). }
}$$
\end{example}

\begin{example} \label{Ex:cube3} Consider the following quiver
$$\vcenter{\xymatrix@R=0ex{
3 \ar@{_<-^>}[rr]^{b_2,b_1,b_{-2},b_{-1}}_{b_{-2}^r,b_{-1}^r} && 2  \ar@{_<-^>}[rr]^{a_1,a_{-1}}_{a_{-1}^r} && 1}
} $$
with relations \begin{align*}
& a_{\epsilon} a_{-1}^r=e_1,\quad b_{i\epsilon} b_i^r=e_2, \\
& a_{\epsilon}b_{2\epsilon}=a_{\epsilon}b_{\epsilon},\ a_{-\epsilon}b_{2\epsilon}=a_{\epsilon}b_{-\epsilon}, \\
& b_{-2}^r a_{-1}^r=b_{-1}^r a_{-1}^r,\\
& b_{2\epsilon} b_{-1}^r = b_{\epsilon} b_{-2}^r = a_{-1}^r a_{\epsilon},\quad \text{where $\epsilon=\pm 1$}.
\end{align*}
If we choose the basic set to be $$\{e_{31}=e_3-b_{-2}^rb_{-2}-b_{-1}^rb_{-1}+b_{-2}^rb_{-2}b_{-1}^rb_{-1},e_{21}=e_2-a_{-1}^ra_{-1},e_{11}=e_1\},$$
then $A^b$ is the algebra presented by
$$\vcenter{\xymatrix@R=0ex{
3 \ar[rr]^{b_2', b_1'} && 2  \ar[rr]^{a_1'} && 1}
} $$
where \begin{align*}& b_1'=e_{21}b_1e_{31}=b_1-b_{-1}-b_1b_{-2}^rb_{-2}+b_{-1}b_{-2}^rb_{-2}, \\
& b_2'=e_{21}b_2e_{31}=b_2-b_{-2}-b_2b_{-1}^rb_{-1}+b_{-2}b_{-1}^rb_{-1},\\
& a_1'=e_{11}a_1e_{21}=a_1-a_{-1}. \end{align*}
Interested readers can verify that this set of idempotents is obtained from the breaking at the sequence $(b_{-2},a_{-1},b_{-1})$ or $(a_{-1}b_{-1},b_{-2},b_{-1},a_{-1})$.
Note that $a_1'b_1'=a_1'b_2'$ gives the relation of $A^b$.
So $\nu(M)$ can be simplified as:
\begin{align*}\vcenter{\xymatrix{
\Coker(b_{-2}^r-b_{-1}^r) \ar[rr]^{\quad b_1-b_{-1},}_{\quad b_2-b_{-2}} && \Coker(a_{-1}^r) \ar[rr]^{a_1} && M(1) }
}\\
\vcenter{\xymatrix{
\Ker(b_{-1})\cap\Ker(b_{-2}) \ar[rr]^{\qquad b_2,b_1} && \Ker(a_{-1}) \ar[rr]^{a_1} && M(1) }
}\end{align*}
The inverse of $\nu$ is represented by the projective representation $P=\oplus\sm{1&2&1\\0&1&1\\0&0&1}\sm{P_3\\P_2\\P_1}$.
\end{example}

\begin{example} \label{Ex:cubewp3} We add a loop to the previous example:
$$\vcenter{\xymatrix@R=0ex{
3 \ar@{_<-^>}[rr]^{b_2,b_1,b_{-2},b_{-1}}_{b_{-2}^r,b_{-1}^r} \ar@(ul,ur)|{t} && 2  \ar@{_<-^>}[rr]^{a_1,a_{-1}}_{a_{-1}^r} && 1}
} $$
satisfying additional relations \begin{align*}
t^2=e_3,\quad b_{i\epsilon}t=b_{j\epsilon},\quad tb_{i\epsilon}^r=b_{j\epsilon}^r,\ (i\neq j).
\end{align*}
Choose the same basic set, then $A^b$ is the algebra presented by
$$\vcenter{\xymatrix@R=0ex{
3 \ar[rr]^{b_1'}  \ar@(ul,ur)|{t'} && 2  \ar[rr]^{a_1'} && 1}
} $$
with relations
$${t'}^2=e_3,\ a_1'b_1'=a_1'b_1't'.$$
We a similar description for $\nu(M)$.
Note that this algebra is of finite representation type over a field.
\end{example}

\begin{definition} \label{D:GReedy} A {\em degree function} on a $(Q,R,S)$ is an injective map $d:Q_V\to \B{Z}$ such that $d(u)>d(v)$ if there is a section $a:v\to u$. A QRS is called {\em gradable} if it admits a degree function.
A QRS with a degree function is called {\em GReedy} if every path $p$ factors as $p=sr$, where $r$ (resp. $s$) is a composite of arrows not raising (resp. not lowering) the degree.
Moreover, $s$ and $r$ are unique up to some loops.
\end{definition}

The name GReedy comes from the generalized Reedy category considered by topologists \cite{BM}. By induction we can easily see that to verify the existence part of the GReedy condition, it is enough to verify for path of form $p=a_r a_s$, where $a_r$ (resp. $a_s$) is an arrow not raising (resp. not lowering) the degree.
In this notes, unless otherwise stated, we always consider the degree functions given by the vertex numbering of quivers.
We can check that all examples so far are GReedy. Vaguely speaking, if one wants to reduce the complexity of the representation theory of a GReedy QRS, then one can add sections and in the meanwhile keep the GReedy condition.

\begin{example} For example, the quiver $\xymatrix{2  \ar@<0.5ex>[r]^{a}\ar@<-0.5ex>[r]_{b} & 1}$
is trivially GReedy. We add one section $b^r$ such that $bb^r=e_1$, violating the GReedy condition. Its basic algebra is
$\xymatrix{2  \ar[r]^{a} & 1 \ar@(ur,dr)|{ab^r} }$, whose representation theory is even more complicated than the original one. To keep it GReedy, we need additional relation $ab^r=0$ or $ab^r=e_1$, then its basic algebra reduces to a simpler one $\xymatrix{2  \ar[r]^{a} & 1 }$.
\end{example}

\section{Quiver Complexes} \label{S:Qcomp}

Given a quiver with relations $\D{Q}=(Q,R)$, we can associate a $k$-category \cite[Definition A.1.4]{ASS} $\D{P}(\D{Q})$ as follows. The objects are the vertices of $Q$, and the set of morphisms from $u$ to $v$ are $k$-linear combinations of paths in $Q$ from $u$ to $v$. Composition of morphisms in $\D{P}(\D{Q})$ is defined through concatenation of paths. The relations among morphisms inherit the relations in $R$.
Sometimes we treat a quiver with relations $\D{Q}$ directly as the category $\D{P}(\D{Q})$.

\begin{definition} Given a quiver with relations $\D{Q}$, its representation in a $k$-category $\D{C}$ is a $k$-linear covariant functor $\rho:\D{Q}\to \D{C}$. If the algebra $k\D{Q}$ is basic, such a representation is also called a {\em quiver complex}.
For arrows $a,a^r$ with relation $aa^r=e_{ha}$, $\rho(a)$ is called a {\em coface map} or {\em projection}, and $\rho(a^r)$ a {\em face map} or {\em section}.
\end{definition}

\begin{definition} Let $A$ be a $k$-algebra. A representation of $A$ in a $k$-category $\D{C}$ is a $k$-algebra morphism $A\to \End_{\D{C}}(M)$ for some $M\in\D{C}$.
\end{definition}

The following lemma can be proved by repeating word by word the proof of \cite[Theorem III.1.6]{ASS}.
\begin{lemma} Suppose that $A$ is the algebra associated to a quiver with relations $\D{Q}$. The category $\Rep(A,\D{A})$ of representations of $A$ in a $k$-linear category $\D{A}$ is equivalent to the category $\Rep(\D{Q},\D{A})$ of representations of $\D{Q}$ in $\D{A}$.
\end{lemma}

\begin{lemma} Let $A,B$ be two $k$-algebras. The category of representation of $A$ in the category of right $B$-modules is the same as the category of $(A,B)$-bimodules.
\end{lemma}
\begin{proof} The proof goes almost formal. Given any $\rho:A\to\End_B(M)$ for some right $B$-module $M$, we define an $(A,B)$-bimodule on the same $k$-module $M$ by $(a,b)(m)=\rho(a)mb$. Conversely, for any $(A,B)$-bimodule $M$, we get a right $B$-module $M$ by forgetting the action of $A$. The action of $A$ defines $\rho:A\to\End_B(M)$.
\end{proof}

\begin{lemma} \label{L:Meq} Let $A,A'$ be two $k$-algebras. The categories of their representations in any $k$-linear abelian category $\D{A}$ are equivalent if and only if their module categories $\Mod(A)$ and $\Mod(A')$ are equivalent.
\end{lemma}

\begin{proof} $``\Rightarrow"$ is trivial. Conversely if $\Mod(A)\cong\Mod(A')$, then by Morita's theorem that there is a projective right $A$-module $P$ such that $A'^{\op}\cong\End_{A}(P)$ and the equivalence of categories is given by the functor $_{A'}P_{A}\otimes -$. Using Freyd-Mitchell Embedding \cite[1.6.1]{We}, we embed $\D{A}$ into the category of $B$-modules for some $k$-algebra $B$. Then the category of representation of $A$ in $\Mod(B)$ is the same as the category $A\otimes B^{\op}$-modules.
We consider $B\otimes P$ as a right $A\otimes B^{\op}$-module, or equivalently a $(B,A)$-bimodule.
Since by adjunction $\Hom_{A\otimes B^{\op}}(B\otimes P,-)$ is
\begin{align*}  \Hom_{(B,A)}(B\otimes k\otimes P,-)=\Hom_{B}(B\otimes k,\Hom_A(P,-))=\Hom_A(P,-),\\
\intertext{$B\otimes P$ is projective with}
\End_{(B,A)}(B\otimes P)=\Hom_A(P,B\otimes P)=B\otimes \End_A(P)=B\otimes A'^{\op}.
\end{align*}
We use Morita's theorem again to conclude that their categories of representations in $\Mod(B)$ are equivalent. This equivalence clearly restricts to representations in $\D{A}$.
\end{proof}

Since every abelian category is at least $\B{Z}$-linear, we mostly take the base ring to be $\B{Z}$. To simplify our argument, from now on we will assume our abelian category $\D{A}$ to be the category of $B$-modules. Let $A$ be the algebra corresponding to a quiver with relations, and denote $T:=A\otimes B^{\op}$. We will often think any $T$-module $V$ as a functor, either as $V\otimes -:\Mod(B)\to\Mod(A)$ or as $\Hom_A(-,V):\Mod(A)\to\Mod(B)$, but we prefer the latter. Let $P_1\xrightarrow{f} P_0$ be a map between two projective modules of $A$, then $f$ is a matrix with each entry a linear combination of paths. The application of the functor $\Hom_A(-,V)$ to $f$ can be viewed as an evaluation. Concretely, if $P_1$ and $P_0$ corresponds to vertex $v$ and $u$ respectively, and $f$ a path from $u$ to $v$, then $\Hom_A(f,V)$ is the map $V(f):e_u V \to e_v V$.

Let $\nu:\Mod(A)\to\Mod(A^b)$ be the (normalized) quiver complex functor. We denote by $\nu_{\D{A}}$ the equivalence induced from $\nu$ as in Lemma \ref{L:Meq}. Let $q$ still be the quotient functor $\Mod(A^b)\to \underline{\Mod}(A^b)$.

\begin{definition}
The composite $\underline{\pi}_{\D{A}}=q \nu_{\D{A}}$ is called the {\em Q-homotopy functor in $\D{A}$}. For any $T\in\Mod(A^b)$, the {\em $i$-th classical homotopy relative to $T$}, or {\em $T$-homotopy in $\D{A}$} is the functor $\pi_i(T,-):=\Ext_{A^b}^i(T,\nu_{\D{A}}(-))$. The cohomotopy functors are defined analogously.
\end{definition}

First we want to explain how the above definitions generalize the classical theory.
The truncated version of classical Dold-Kan correspondence says that the category of $n$-truncated simplicial objects in $\D{A}$ is equivalent to the category of chain complexes in $\D{A}$ concentrated in degree $1,2,\dots,n$.
The former is the category of representations in $\D{A}$ of the algebra $A$:
$$\BAnpi{}{}{}{}{}{}$$
satisfying the {\em simplicial relations}:
\begin{align} \label{eq:simprel1}
& \pi_j \pi_i = \pi_{i-1} \pi_j\quad \text{ if } j<i, \\
& \pi_j \iota_i= \begin{cases} \iota_{i-1} \pi_j & \text{if } j<i, \\
e_* & \text{if } j=i \text{ or } i+1, \\
\iota_{i} \pi_{j-1} & \text{if } j>i+1, \end{cases}\\
& \iota_j \iota_i = \iota_{i} \iota_{j-1}\qquad \text{ if } j>i.
\label{eq:simprel2} \end{align}
The latter is the category of representations in $\D{A}$ of the algebra of complexes $A^b$:
$$\An{n}{d}{n-1}{2}{d}{1}{d}{d}$$
The untruncated Dold-Kan correspondence can be obtained by applying the obvious colimit to the $n$-truncated ones.
The reason why we consider the truncated version first is that we want to avoid infinite-dimensional algebras.
From now on, we will treat the truncated cases only.

\begin{proposition} \label{P:CDK} The truncated Dold-Kan correspondence is induced from the functor $R=\Hom_{A^b}(P,-):\Mod(A^b)\to\Mod(A)$, where $P=\oplus\B{P}_nT$, and $T$ is an upper-triangular integer matrix $T$ obtained from Jia Xian's Triangle: $T_{ij}=\sm{j\\i}$ for $1\leq i\leq j \leq n$.
\end{proposition}
Moreover, since we have resolutions:
\begin{align*}& 0\to P_1\to P_2\to \cdots \to P_n\to S_n\to 0;\\
& 0\to S_1\to I_1\to I_2\to \cdots \to I_n\to 0.
\end{align*}
The ($n$-truncated) classical homology and cohomology are related to the $S_n$-homotopy and $S_1$-cohomotopy respectively.

\begin{proof} If we carefully inspects the proof of the classical Dold-Kan correspondence (eg. \cite[8.4.4]{We}), we can find that it already contains our statement. Here we give a slightly different proof.

It is not hard to verify that after breaking at the sequence $\mathbf{S}_1=(\pi_{1},\dots,\pi_{1},\pi_{1})$ from the leftmost to the rightmost, the quiver of the new algebra is obtained from the old one by simply removing all arrows $\pi_1$ and $\iota_1$. The new quiver has the same relations as the old one. Next we break at the sequence $\mathbf{S}_2=(\pi_{2},\dots,\pi_{2},\pi_{2})$ from the leftmost to the rightmost, then we keep the same except that we remove all arrows $\pi_2$ and $\iota_2$ and add one extra relation from vertex $3$ to $1$: $\pi_2\pi_3=0$.

By induction, consecutive breakings at sequences $(\mathbf{S}_{n-1},\dots,\mathbf{S}_2,\mathbf{S}_1)$ give the algebra of complexes, where $S_i$ is the sequence $(\pi_i,\dots,\pi_i,\pi_i)$ from the leftmost to the rightmost. Now everything follows from Theorem \ref{T:Br}. Note that $D_{n-1}\cdots D_2D_1=T$, where $D_i$ is the matrix obtained from the $n\times n$ identity matrix by adding ones on the $i$-th upper diagonal.
\end{proof}

It is well-known that the category of cyclic objects is related to the category of mixed complexes \cite[9.8]{We}.
The former is the category of representations $\D{A}$ of the following path algebra $A$:
$$\BAnpi{}{\ar@(ul,ur)|{t}}{}$$
satisfying the simplicial relations \eqref{eq:simprel1}--\eqref{eq:simprel2} and {\em cyclic relations}
\begin{align} \label{eq:cycrel1} & \pi_i t=t\pi_{i-1},\quad \iota_i t=t\iota_{i-1} \\
\label{eq:cycrel2} & \pi_1t = \pi_n,\quad \iota_1 t = t^2 \iota_n.
\end{align}
The latter is the category of representations in $\D{A}$ of the algebra $A^b$:
$$\mixedcomp{d}{B}$$
with relations
$$d^2=0,B^2=0,\text{  and  }dB+Bd=0.$$
They are almost but not exactly equivalent as we have seen in Example \ref{Ex:cycle3}.
In fact, the category of mixed complexes is equivalent to the category of {\em symclic objects}. A symclic object is something lies between a simplicial
object and a cyclic object:
$$\BAnpi{\iota,}{}{}$$
where $\iota:k-1\to k$ is $(-1)^k(t\iota_k)N$ and $N=e_*+\sum_{i=1}^{k-2}(-1)^{ik}t^i$. The proof of the following proposition is similar to that of Proposition \ref{P:CDK}.

\begin{proposition} \label{P:SDK} The above equivalence is induced from the functor $R=\Hom_{A}(P,-):\Mod(A^b)\to\Mod(A)$, where $P$ has the same formula as in Proposition \ref{P:CDK}. Moreover, since we have resolution:
$$\vcenter{\xymatrix@C=3ex@R=3ex{
0\ar[r] & P_1 \ar[r] & \cdots \ar[r] & P_{n-2} \ar[r]\ar[dr] \ar@{}[d]|{\oplus} & P_{n-1}\ar[r]\ar[dr]\ar@{}[d]|{\oplus} & P_{n}\ar[r]\ar@{}[d]|{\oplus} & P_{n-1}\ar[r]\ar@{}[d]|{\oplus} & P_{n-2}\ar[r] \ar@{}[d]|{\oplus} &\cdots \ar[r]  & P_1 \ar[r] & S_1\ar[r] & 0 \\
& & & \cdots\ar[r] & P_{n-3}\ar[r]\ar[dr] \ar@{}[d]|{\oplus} & P_{n-2}\ar[r]\ar[ur]\ar@{}[d]|{\oplus} & P_{n-3}\ar[r]\ar[ur] \ar@{}[d]|{\oplus} & \cdots & &  \\
& & & & \cdots\ar[r] &  P_{n-4} \ar[r]\ar[ur]\ar@{}[d]|{\oplus} & \cdots & & & \\
& & & & & \vdots & & & & &
}} $$
The ($n$-truncated) cyclic cohomology \cite[Proposition 9.8.3]{We} is nothing but the $S_1$-homotopy.
\end{proposition}

We also have the resolution for $0\to S_{n} \to P_{n}\to\cdots \to P_3\to P_2\to I_1\to 0$, so the $I_1$-homotopy is the (shifted) simplicial cohomology. Note that $S_2$ is the syzygy of $S_1: 0\to S_2\to P_1\to S_1\to 0$.
From the exact sequence $0\to S_1\to I_1\to S_2\to 0$, we obtain Connes's SBI sequence \cite[Proposition 9.6.11]{We} up to $\pi_{n-2}$:
$$\cdots\to\pi_{i-1}(S_1,-)\to\pi_i(S_2,-)\to\pi_i(I_1,-)\to\pi_i(S_1,-)\to\pi_{i+1}(S_2,-)\to\cdots.$$
Sensitive readers must find that the above method is related to the spectral sequence algorithm. In fact, this idea elaborated in \cite{F3} can be applied to any quiver complex (not necessary double complex). We refer the readers to Example \ref{Ex:Dn1dl} for another example.

A {\em cubical objects} in $\D{A}$ is a representation in $\D{A}$ of the algebra $A$
$$\BSnpi{}{}{}$$ satisfying the {\em cubical relations}:
\begin{align} \label{eq:cuberel1}
& \pi_{\epsilon i}\pi_{\delta j}=\pi_{\delta(j-1)}\pi_{\epsilon i} & \text{ if } i<j,\\
& \iota_i\iota_j = \iota_{j} \iota_{i-1} & \text{ if } i> j, \\
& \pi_{\epsilon i} \iota_j = \begin{cases} \iota_{j-1}\pi_{\epsilon i} & \text{ if } i<j \\  \iota_{j}\pi_{\epsilon(i-1)} & \text{ if } i>j \\ e_* & \text{ if } i=j
\label{eq:cuberel2} \end{cases}  & \text{ where $\delta$ and $\epsilon$ are $\pm 1$.} \end{align}
A {\em cubical object with permutations} in $\D{A}$ is a representation in $\D{A}$ of the algebra $A$
$$\BSnpiwp{}{}$$ satisfying the additional relations including the relations of $t_i$'s as the $i$-th transposition $(i,i+1)$ in the symmetric group $S_n$, and
\begin{align*}
\pi_{i\epsilon}t_{i}=\pi_{(i+1)\epsilon},\quad t_{i}\iota_{i\epsilon}=\iota_{(i+1)\epsilon}.
\end{align*}

\begin{proposition} \label{P:CuDK} The category of cubical objects in $\D{A}$ is equivalent to the category of representation $\D{A}$ of the algebra
$A^b$:
$$\An{n}{\mathbbm{d}_{n-1}}{n-1}{2}{d_1}{1}{\mathbbm{d}_{n-2}}{\mathbbm{d}_2}$$
satisfying the relations from vertex $k+1$ to $k-1$: $d_{j+1}d_i=d_jd_i$, where $d_{k+1}=d_1$.
The category of cubical objects with permutations in $\D{A}$ is equivalent to the category of representation $\D{A}$ of the algebra
$A^b:$
$$\Anwp{n}{d}{n-1}{d}{3}{d}{2}{d}{1}$$
satisfying the relations from vertex $k+1$ to $k-1$: $d_{j+1}d_i=d_jd_i$, where $d_i=dt_1t_2\cdots t_{i-1}$ and $d_{k+1}=d_1$.
Both equivalence is induced from the functor having the same formula as in Proposition \ref{P:CDK}.
\end{proposition}

\begin{proof} The proof is similar to that of Proposition \ref{P:CDK}. We consider the breakings at sequences $S_{n-1},\dots,S_2,S_1$, where $\mathbf{S}_i=(\pi_{-i},\dots,\pi_{-i},\pi_{-i})$ from the leftmost to the rightmost.
\end{proof}


\section{Orthogonal Projections} \label{S:OP}

In this section, we take our base ring $k$ to be a field. We believe that many key results have analogs for $k=\B{Z}$, but one has to work much harder. For our purpose, we do not need results in that generality.
Let $Q$ be a finite quiver possibly with oriented cycles.
By abuse of notation, we write $\Mod(Q)$ for the category of finitely generated left modules of $kQ$.

For two $kQ$-modules $M$ and $N$, $M$ is said to be left {\em orthogonal} to $N$ denoted by $M\perp N$ if $\Hom_Q^\bullet(M,N):=\Hom_Q(M,N)\oplus\Ext_Q(M,N)=0$. In this case we also say that $N$ is right orthogonal to $M$. Let $\D{C}$ be a collection of modules. The right {\em orthogonal category} $\D{C}^\perp$ is the abelian subcategory $\{N\in\Mod(Q)\mid M\perp N, \forall M\in\D{C}\}$. Let $\innerprod{\D{C}}$ be the abelian subcategory generated by $\D{C}$. It is easy to verify that $\D{C}^\perp=\innerprod{\D{C}}^\perp$.

Suppose $E\in\Mod(Q)$ is exceptional, i.e., $\Hom_Q^\bullet(E,E)$ is $1$-dimensional generated by the identity morphism, so the dimension vector of $E$ corresponds to a {\em real Schur root} \cite[1]{S2} $\epsilon$. Moreover, we assume that $E$ is {\em right and left Hom-finite}, i.e. $\Hom_{Q}^\bullet(E,X)$ and $\Hom_{Q}^\bullet(X,E)$ are finite dimensional for all $X\in\Mod(Q)$.
We specialize some general results in \cite{GL} (see also \cite{S1}) to the quiver case.

\begin{lemma} \cite[Proposition 3.2, 3.5]{GL} \label{L:OP} $E^\perp$ is a reflective subcategory of $\Mod(Q)$, i.e.,
there is a functor $\tilde{\pi}_E:\Mod(Q)\to E^\perp$ left adjoint to the inclusion functor
$\iota_E:E^\perp\to\Mod(Q)$. In particular, $\tilde{\pi}_E$ is right exact and compatible with projective
presentations.
\end{lemma}

\begin{proof}[Sketch of proof.] It is useful to recall the construction in \cite{GL}, which is depicted by the following diagram. The row is the {\em universal extension} and the column is the {\em universal homomorphism}.
$$\xymatrix{
& E^h \ar[d] & \\
M \ar@{^{(}->}[r] \ar[dr]_{r_M} & M' \ar@{->>}[d] \ar@{->>}[r] & E^e\\
& \tilde{\pi}_E(M) &
}$$
Here, the universal extension means the extension universal with respect to the property that the connecting morphism $\Hom_Q(E,E^e)\xrightarrow{\delta}\Ext_Q(E,M)$ in the long exact sequence is an isomorphism. Similarly, we mean by the universal homomorphism.

Alternatively, we can change the order of taking the universal extension and the universal homomorphism, but this leads to the same construction.
$$\xymatrix{
E^h \ar[r] & M \ar[dr]_{r_M} \ar@{->>}[r] & M'' \ar@{_{(}->}[d] \\
& & \tilde{\pi}_E(M) \ar@{->>}[d] \\
& & E^e
}$$
Note that the composition $r_M$ is the universal homomorphism from $M$ to an object of $E^\perp$.
\end{proof}

Let $F$ be another exceptional object. In general, $\tilde{\pi}_F(E)$ may not be exceptional, or even indecomposable. However, we have Theorem \ref{T:rel} due to the following lemma.

\begin{lemma} \label{L:epimono} \cite[Lemma VIII.3.3]{ASS} If $M$ and $N$ are indecomposable modules such that $\Ext_Q(M,N)=0$, then any nonzero homomorphism from $N$ to $M$ is either a monomorphism or an epimorphism.
\end{lemma}

\begin{definition} If $\tilde{\pi}_E(M)=0$, then the universal homomorphism $E^h\to M$ is surjective, and we define $\tilde{\varpi}_E(M)$ to be the kernel of $E^h\twoheadrightarrow M$. Otherwise we define $\tilde{\varpi}_E(M)=\tilde{\pi}_E(M)$. Note that $E\perp \tilde{\varpi}_E(M)$.
\end{definition}

\begin{theorem} \label{T:rel} If $E\perp F$, then $E'=\tilde{\varpi}_F(E)$ is exceptional, and $\innerprod{F,E'}$ generates the same abelian subcategory as $\innerprod{E,F}$. Moreover, we have that
\begin{equation} \label{eq:rel1} \tilde{\pi}_{F}\tilde{\pi}_E=\tilde{\pi}_{E'}\tilde{\pi}_{F},\quad \iota_E \iota_{F} = \iota_{F}\iota_{E'},\quad \text{and}\quad \tilde{\pi}_{F}\iota_{E}=\iota_{E'}\tilde{\pi}_{F}.\end{equation}
\end{theorem}

\begin{proof} We believe that the first statement is well-known. With the help of Lemma \ref{L:epimono}, it can be proved by playing homological algebra in the construction of Lemma \ref{L:OP}.

For the last statement, we observe that the codomain of $\tilde{\pi}_{E'}\tilde{\pi}_{F}$ is $(F,E')^\perp$. But $(F,E')^\perp=(E,F)^\perp$, which is the codomain of $\tilde{\pi}_{F}\tilde{\pi}_{E}$. Now according to Yoneda embedding, to show that $\tilde{\pi}_{F}\tilde{\pi}_E=\tilde{\pi}_{E'}\tilde{\pi}_{F}$, it is enough to show that $\Hom_Q(\tilde{\pi}_{F}\tilde{\pi}_E(M),N)=\Hom_Q(\tilde{\pi}_{E'}\tilde{\pi}_{F}(M),N)$ for any $N$ in $(E,F)^\perp$. But both sides equal to $\Hom_Q(M,N)$ due to the adjunction.

The equality $\iota_E \iota_{F} = \iota_{F}\iota_{E'}$ is trivial. $\iota_{E'}\tilde{\pi}_{F}$ considered as a functor defined on $E^\perp$, has the same codomain as $\tilde{\pi}_{F}\iota_{E}$. Similar argument as before shows that $\tilde{\pi}_{F}\iota_{E}=\iota_{E'}\tilde{\pi}_{F}$.

$$\vcenter{\xymatrix@R=2ex@C=16ex{
&  E^\perp  \ar@{_<-^>}[dr]^{\tilde{\pi}_F}_{\iota_F}  \\
\Mod(Q)	\ar@{_<-^>}[ur]^{\tilde{\pi}_E}_{\iota_E}  \ar@{^<-_>}[dr]_{\tilde{\pi}_F}^{\iota_F} & & (E,F)^\perp 	 \\
&  F^\perp  \ar@{^<-_>}[ur]_{\tilde{\pi}_{E'}}^{\iota_{E'}}  	}
} $$
\end{proof}

\begin{definition} The relations in \eqref{eq:rel1} are called {\em fundamental relations of first kind}. The relations of form $\tilde{\pi}_E\iota_F=\sigma$ for some equivalence $\sigma$, is called {\em fundamental relations of second kind}.
\end{definition}

\begin{definition} An {\em exceptional collection} is a set $\B{E}$ of exceptional objects such that for any $E,F\in\B{E}$ we have either $E\perp F$ or $F\perp E$.
\end{definition}

Any {\em exceptional sequence} is an exceptional collection.

\begin{definition} The {\em basic set} $b(M)$ of a module $M$ consists of all the non-isomorphic direct summands of $M$. We call $M$ {\em basic} if its direct summands are all pairwise non-isomorphic.

A $kQ$-module $T$ is called {\em partial tilting} if $\Ext_Q(T,T)=0$. A partial tilting module is called {\em tilting} if its direct summands generate the category $\Mod(Q)$.
This is equivalent to say that the cardinality of the basic set of $T$ is equal to the number of vertices $|Q_V|$ \cite[Proposition VI.4.4]{ASS}.
\end{definition}

\begin{remark} \label{R:proj} If $P_v$ is the indecomposable projective module corresponding to a vertex $v$, then
applying $\tilde{\pi}_E$ to $P_v$ is particularly simple. If $E$ is not projective, then $\Hom_Q(E,P_v)=0$; otherwise $\Ext_Q(E,P_v)=0$.
Clearly, $\tilde{\pi}_E(kQ)=\bigoplus_{v\in Q_V}\tilde{\pi}_E(P_v)$ is partial tilting. Applying $\Hom_Q(-,E)$ to the construction of Lemma \ref{L:OP}, we see that $\Ext_Q(\tilde{\pi}_E(kQ),E)=0$, so $\tilde{\pi}_E(kQ)\oplus E$ is a tilting module for $\Mod(Q)$. Let $P^b$ be a direct sum of elements in $b(\tilde{\pi}_E(kQ))$, then the quiver of $\End_Q(P^b)$ has $|Q_V|-1$ vertices.
\end{remark}

\begin{corollary} \cite[Theorem 2.3]{S1} \label{C:OC}
The orthogonal category $E^\perp$ is equivalent to representations of a quiver $Q_E$ with one vertex less than $Q$. The inverse functor $\iota_E: \Mod(Q_E)\to E^\perp$ is a fully faithful exact embedding into $\Mod(Q)$.
\end{corollary}

\begin{definition} An exceptional representation $E$ of $Q$ is called {\em connected} if the quiver $Q_E$ is connected.
\end{definition}

We call the functor $\tilde{\pi}_E$ in Lemma \ref{L:OP} the (right) {\em orthogonal projection through} $E$ and its composition with the equivalence in Lemma \ref{C:OC} the (right) orthogonal projection {\em to} $Q_E$, denoted by $\pi_E:\Mod(Q)\to\Mod(Q_E)$. Sometimes we do not distinguish $E^\perp$ and $\Mod(Q_E)$ if no confusion is possible. We should understand that $\pi_E$ is determined only up to automorphisms of $\Mod(Q_E)$.

Let us briefly review the $k$-linear automorphism group $\Pic_k(Q)$ of $\Mod(Q)$. Here, $k$-linear means that those automorphisms induce homomorphisms of $k$-modules between the $\Hom$ groups. It is known that $\Pic_k(Q)\cong \Aut_k(Q)/\Inn_k(Q)$, where $\Aut_k(Q)$ is the $k$-algebra automorphism group of $kQ$, and $\Inn_k(Q)$ is the normal subgroup of inner automorphisms. The latter is nothing but the torus $T_{in} =(k^*)^{Q_V}/k^*$ with $k^*$ embedded multi-diagonally.
Let $\Aut_0(Q)$ be the subgroup of $\Aut_k(Q)$ containing vertex permutations and $\Aut_1(Q)$ be the group $\prod_{u,v\in Q_V}\GL(a_{uv})$ acting naturally on the space of arrows from $u$ to $v$, then $T_{in}$ is also contained in $\Aut_1(Q)$.
Clearly we have that
$$\Aut_k(Q)=\Aut_0(Q)\times \Aut_1(Q).$$
We also remark that any exceptional $E$ is fixed by $\Aut_1(Q)$. This is because $\Aut_1(Q)$ also acts on the representation spaces $\Rep_\epsilon(Q)$ and $E$ is {\em rigid} there. The automorphisms that we mainly consider are in the finite subgroup $\Aut_0(Q)\times N$ of $\Pic_k(Q)$, where $N$ is the Weyl group of $\Aut_1(Q)$ permuting arrows.

Everything above has a dual statement for the left orthogonal category $^\perp E$ and left orthogonal projection $_E\tilde{\pi}$ and $_E\pi$.
Let $\tau$ be the classical AR-transformation \cite[IV.2]{ASS} on $\Mod(Q)$. If $E$ is not projective, then $E^\perp=^\perp\tau E$ by the AR-duality \cite[Theorem IV.2.13]{ASS}. We define the dual right orthogonal projection $\tilde{\pi}_E^\vee:=_{\tau E}\tilde{\pi}$, the left orthogonal projection through $\tau E$. If $E=P$ is projective, then $P^\perp=^\perp \nu P$, where $\nu$ is the Nakayama functor \cite[Definition III.2.8]{ASS}. We define the dual right orthogonal projection $\tilde{\pi}_P^\vee:=_{\nu P}\tilde{\pi}$.

Let $I=\oplus_{v\in Q_V}I_v=kQ^{\op}$, then the dual of Lemma \ref{L:OP} implies that $b(\tilde{\pi}_E^\vee(I))$ contains all the indecomposable injective module in $E^\perp$. We denote $I^b$ the direct sum of all elements in $b(\tilde{\pi}_E^\vee(I))$.
By Eilenberg-Watts theorem, an adjoint pair between module categories must be representable by a bimodule.
The next lemma says that the bimodule for $\tilde{\pi}_E$ is explicitly given by $\tilde{\pi}_E^\vee(I)$.

\begin{lemma} \label{L:TD} \cite[Lemma 2.5]{F1} $\tilde{\pi}_E(M)=\Hom_Q(-,\tilde{\pi}_E^\vee(I))^*$, so $\pi_E(M)=\Hom_Q(-,I^b)^*$.
\end{lemma}

We can naturally order $b(\tilde{\pi}_E^\vee(I))$ by their $\delta$-vectors in $K(\Inj\-kQ)$ \cite{DF}. This provides us an {\em canonical ordering} on the vertices of the new quiver $Q_E$. The functor $\pi_E$ can be readily described using the injective resolutions of each one in $b(\tilde{\pi}_E^\vee(I))$. For example,

\begin{lemma} \label{L:sproj} For a non-projective simple $S_i$, we have that
$$\tilde{\pi}_i^\vee(I_j)=\begin{cases}\bigoplus_{k\in Q_V} a_{ik}I_k & \text{ if  } i=j, \\ I_j & \text{ otherwise,}\end{cases}$$
where $a_{ik}=\ext_Q(S_i,S_k)$ is the number of arrows from $i$ to $k$.
When the simple $S_i$ is projective,  $\tilde{\pi}_i^\vee(I_j)=I_j$ if $i\neq j$, otherwise equals zeros.
In particular, $\pi_i$ is an exact functor. In fact, $\pi_E$ is exact if and only if $E$ is some $S_i$.
\end{lemma}

\begin{proof} If $S_i$ is not projective, $\Hom_Q(I_j,\tau S_i)\cong\Ext_Q(S_i, I_j)$ always vanishes. If $i\neq j$, then $\Ext_Q(I_j,\tau S_i)\cong\Hom_Q(S_i, I_j)$ is zero as well.
We have the canonical resolution of $S_i$: \label{eq:resimp}
\begin{equation}0\to \bigoplus_{a:i\to k} P_k\to P_i\to S_i \to 0.\end{equation}
So $$0\to \tau S_i \to  \bigoplus_{a:i\to k} I_k\to I_i\to 0.$$
Since $\Ext_Q(I_i,\tau S_i)=k$, this sequence is the universal extension in the construction of Lemma \ref{L:OP}, and we conclude that $\tilde{\pi}_i^\vee(I_i)=\bigoplus_{k\in Q_V} a_{ik}I_k$.  When $S_i$ is projective, the formula is clear. The exactness of $\pi_i$ follows from Lemma \ref{L:TD}.

For the last statement, we suppose that $\pi_E$ is exact but $E$ is not any $S_i$. If $E$ is not projective, then we have at least two injective resolutions coming from the universal extensions:
\begin{align*}
0\to \epsilon_i(\tau E) \to  \tilde{\pi}_E^\vee(I_i)\to I_i\to 0, \quad 0\to \epsilon_j(\tau E) \to  \tilde{\pi}_E^\vee(I_j)\to I_j\to 0.
\end{align*}
The middle terms are injective because $\pi_i$ is exact. We thus obtain two epimorphisms:
$P_i\twoheadrightarrow \epsilon_i E$ and $P_j\twoheadrightarrow \epsilon_j E$, which contradicts the uniqueness of the projective cover.
If $E=P$ is projective, we get at least one sequence: $0\to\tilde{\pi}_E^\vee(I_i)\to I_i\to \nu E$, which is impossible as well.
\end{proof}

\begin{corollary} \label{C:sproj} The functor $\pi_i$ is the following ``contraction" at $i$:
\begin{align*}& \vcenter{\xymatrix@R=3ex{
\ar[dr]_{\vdots}^{A_1} & & \\
& i \ar[ur]_{\vdots}^{B_1} \ar[dr]_{B_n}	 \\
\ar[ur]_{A_m} & &	} }
\qquad\qquad\quad\xrightarrow{\pi_i}
& \vcenter{\xymatrix@R=5ex{
\ar@<0.5ex>[rr]^{A_1 B_1} \ar@<-0.5ex>[rr]_{\vdots}|{A_1 B_2} & & \\
\ar[rr]_{A_m B_n} & &	} }
\end{align*}
The functor $\iota_i$ can be described as follows:
\begin{align*}& \vcenter{\xymatrix@R=3ex{
\ar[dr]_{\vdots}^{D_{1}} & & \\
& i \ar[ur]_{\vdots}^{E_1} \ar[dr]_{E_n}	 \\
\ar[ur]_{D_{m}} & &	} }
\qquad\qquad\quad\xleftarrow{\iota_i}
& \vcenter{\xymatrix@R=5ex{
\ar@<0.5ex>[rr]^{C_{11}} \ar@<-0.5ex>[rr]_{\vdots}|{C_{12}} & &\\
\ar[rr]_{C_{mn}} & &	} }
\end{align*}
where the vector space at $i$ is a direct sum of all target spaces of $C_{ij}$ (without repetition), the matrix $D_i=(C_{i1},C_{i2},\dots,C_{in})$, and $(E_1,E_2,\dots,E_n)$ is the identity matrix.
When $i$ is a source (resp. sink), $\pi_i$ should be understood as forgetting all $B_k$'s (resp. $A_k$'s); and $\iota_i$ is also obvious.
Moreover, we have that \begin{align*}
&\pi_i(S_j)=\begin{cases} S_j \\ S_{j-1} \\ 0
\end{cases}
& \pi_i(P_j)=\begin{cases} P_j \\ P_{j-1}  \\  \bigoplus_{k\in Q_V} a_{ik}P_k
\end{cases}
& \pi_i(I_j)=\begin{cases} I_j & \text{ if } i>j, \\ I_{j-1} & \text{ if } i<j,  \\  \bigoplus_{k\in Q_V} a_{ik}I_k  & \text{ if } i=j.
\end{cases}
\end{align*}
\end{corollary}

\begin{proof} Now $b(\tilde{\pi}_i^\vee(I))$ contains all indecomposable injective modules of $\Mod(kQ)$ but $I_i$. The description of $\pi_i$ follows from Lemma \ref{L:TD} immediately.
The last statement about $\pi_i(S_j)$ is clear from this description.
We knew in priori that each $\pi_i(P_j)$ is projective, and indecomposable for $i\neq j$ by Theorem \ref{T:rel}. So for $i\neq j$, the result follows also from Lemma \ref{L:TD}. For $i=j$, we apply $\pi_i$ to the exact sequence \eqref{eq:resimp} and get
$$0\to \pi_i(\bigoplus_{a:i\to k} P_k)=\bigoplus_{a:i\to k} P_k\to \pi_i(P_i) \to \pi_i(S_i)=0.$$
The argument for $\pi_i(I_j)$ is similar.
\end{proof}

\begin{corollary} \label{C:rel2} Let $F=\tau^{-1}S_i$ if $S_i$ is not injective, otherwise let $F=P_i$. Then we get a relation of second kind: $\tilde{\pi}_i\iota_F=\sigma$ is an equivalence.

\end{corollary}

\begin{proof} We first show that $\sigma$ is fully faithful. By adjunction $\Hom_Q(\sigma(M),\sigma(N))=\Hom_Q(M,\sigma(N))$.
Since $S_i$ is simple, the universal homomorphism is $hS_i\to N$ is injective. We can show that $\Hom_Q(M,\sigma(N))=\Hom_Q(M,N)$ by playing homological algebra in the construction of Lemma \ref{L:OP} and using the fact that $F^\perp=^\perp S_i$.

Next, we show that $\sigma$ is an embedding. Suppose that $\sigma(N_1)=\sigma(N_2)$. By fully-faithfulness, let $f\in\Hom_Q(N_1,N_2)$ correspond to the identity morphism in $\Hom_Q(\sigma(N_1),\sigma(N_2))$. We claim that $f$ is an isomorphism. If not, it has a kernel or cokernel, say a cokernel $C$. We apply exact $\sigma$ to the exact sequence $N_1\xrightarrow{f} N_2\to C\to 0$, and conclude that $\sigma(C)=0$. By the construction of Lemma \ref{L:OP}, this implies that $hS_i\twoheadrightarrow C$. But this is impossible since $F$ is not orthogonal to $S_i$.

Finally, we notice that $\sigma$ is left adjoint to $\pi_F^\vee\iota_i=$$_i\pi\iota_i$, which is also exact. Hence, $\sigma$ preserves projective objects. Since $S_i^\perp$ has the same number of nonisomorphic projective indecomposable modules as $F^\perp$, being a full exact embedding, $\sigma$ must be dense as well.
It is well-known that a fully faithful and dense functor is an equivalence.
\end{proof}

If $E$ is projective and $Q$ has no oriented cycles, then each $\pi_{P_i}(P_j)$ is indecomposable and projective.
It is natural to make the following convention:
$$\pi_{P_i}(P_j)=\begin{cases} P_j & \text{ if } i>j, \\ P_{j-1} & \text{ if } i<j. \end{cases}$$
Under this convention, it is easy to verify that
\begin{lemma} \label{L:pproj} \begin{align*}
& \varpi_{P_i}(S_j)=\begin{cases} S_j & \text{ if } i>j, \\ S_{j-1} & \text{ if } i<j,  \\  \bigoplus_{a:i\to k} P_{k-1}  & \text{ if } i=j.\end{cases}
\end{align*}\end{lemma}

\section{The Category $\Delta(\B{Q})$} \label{S:Delta}

\begin{definition} A {\em coface configuration (with reflection)} $\B{Q}$ in a category $\D{C}$ is a quiver with the following assignment. \begin{enumerate}
\item For each $v\in\B{Q}_V$, we assign an object $\B{Q}(v)\in\D{C}$ (possibly with repetition).
\item For each loop $l\in\B{Q}_A$ on $v$, we assign an automorphism $\B{Q}(l)$ of $\B{Q}(v)$.
\item For each arrow $a: u\to v$ with $u>v$ (resp. $u<v$), we assign a {\em projection} (resp. {\em section}) $\B{Q}(a): \B{Q}(u)\to\B{Q}(v)$.
\end{enumerate}
Here, projections (resp. sections) are certain class of epimorphisms (resp. monomorphisms) in $\D{C}$. If there is no section in the assignment, we call such $\B{Q}$ a coface configuration without reflections.
\end{definition}

\begin{definition} \label{D:DeltaQ}
The {\em underlying quiver} $Q$ of $\B{Q}$ is obtained by forgetting the assignment.
A {\em relation} $p_1-p_2$ in $\B{Q}$ is two paths $p_1,p_2$ such that their evaluation in $\B{Q}$ is equal.
Collecting all such relations, we get a quiver with relations $\D{Q}$ of $\B{Q}$.

Given a coface configuration $\B{Q}$, we can associate it with a category $\Delta(\B{Q})$ such that its quiver with relations (Section \ref{S:Qcomp}) is the same as that of $\B{Q}$.
A coface configuration $\B{Q}$ is called {\em rooted at} $r\in Q_V$ if $r$ is the unique source of $Q$.

\end{definition}

\begin{remark} Fixing a base PID $k$, we have the algebra $k\D{Q}$ associated to $\D{Q}$, also denoted by $k\B{Q}$.
By definition, a covariant functor from $\Delta(\B{Q})$ to a $k$-linear category $\D{A}$ is a representation
of $\D{Q}$ in $\D{A}$, which is equivalent to a representation of $k\B{Q}$ in $\D{A}$.
By a slight abuse of language, we may call such a representation a representation of $\B{Q}$.
\end{remark}

\begin{definition} \label{D:CCQ} The category of category of quiver representations is the category $\D{CC}(Q)$ with objects the module categories $\Mod(Q)$ and morphisms the {\em cocontinuous functors}, i.e., those with a right adjoint. We define the sections in $\D{CC}(Q)$ to be full exact embeddings in $\D{CC}(Q)$, and the projections to be those functors having sections as their right adjoints.
A {\em quiver coface configuration}, or {\em QCC} in short, is a coface configuration in the category $\D{CC}(Q)$. It is called a {\em quiver exceptional configuration}, or {\em QEC} in short if every section is a composition of $\iota_E$'s, and every projection is a composition of $\pi_E$'s for exceptional $E$'s. A QCC is called {\em q-connected} if every vertex-quiver is connected.

\end{definition}

\begin{remark} Note that the embedding $\iota_E$ itself has a right adjoint $\pi_E^\vee$, so it is automatically a morphism in $\D{CC}(Q)$.
By Eilenberg-Watts theorem, a cocontinuous functor $F:\Mod(Q)\to\Mod(Q')$ is represented by a $kQ'$-$kQ$-bimodule $B$, i.e., $F=B\otimes -$. We can also define the category $\D{CC}(Q)$ using {\em continuous functors}, i.e., those with a left adjoint, or equivalently those of form $\Hom_{Q'}(B,-)$. But it is not hard to see from the Yoneda embedding and the Hom-tensor adjunction that the two definitions are dual to each other. Moreover, it is easy to see from the AR-duality that different versions of QEC's in the two definitions are essentially the same.
\end{remark}
By abuse of notation, we may write $Q$ for $\Mod(Q)$, $E$ for $\pi_E$ and $E^r$ for $\iota_E$.


\begin{example} \label{Ex:BAn} The underlying quiver with relations of the following QEC rooted at the quiver $A_3$ is the one considered in Example \ref{Ex:simple3}:
$$\vcenter{\xymatrix@R=0ex{
A_3 \ar@{_<-^>}[r]^{\B{S}_3}_{\B{S}_2^r} & A_2 \ar@{_<-^>}[r]^{\B{S}_2}_{S_1^r}  & A_1 }
} $$
$$\simplextri$$
More generally, the QEC $\B{A}:$
$$\BAn{}{A_}{\B{S}_n}{\B{S}_{n-1}^r}{S_2,S_1}{S_1^r}$$
gives rise to the usual simplex category $\Delta(\B{A})$. We can read off the simplicial relations \eqref{eq:simprel1}--\eqref{eq:simprel2} from Corollary \ref{C:sproj}, \ref{C:rel2} and Theorem \ref{T:rel}. In fact, all relations are fundamental of first kind, except for $S_{i+1}S_i^r=\Id$. These relations follow from Corollary \ref{C:rel2} with the fact that $S_{i+1}=\tau S_i$.

\end{example}

\begin{remark} We can slightly change projections and sections, and get an isomorphic QEC:
$$\BAn{}{A_}{\B{S}_{n-1},P_1}{\B{S}_{n-2}^r,P_1^r}{S_1,P_1}{P_1}.$$
We also have a QEC $\B{A}^\vee$ whose algebra is trivial dual to that of $\B{A}$:
$$\BAn{}{A_}{\B{S}_{[2,n]}}{\B{S}_{n}^r}{S_2}{S_2^r,S_1^r}.$$

We can also extend quivers, to be more precise the module categories, but still define an isomorphic QEC as in Example \ref{Ex:BAndim}.
However, we have an obvious notion of {\em minimal models} to rule out that one.
\end{remark}

\begin{example} \label{Ex:BAncyc} The QEC $\mr{\B{A}}:$
$$\BAn{\ar@(ul,ur)|{t}}{\mr{A}_}{\B{S}_n}{\B{S}_{n-1}^r}{S_2,S_1}{S_1^r}$$
gives rise to the usual cycle category $\Delta(\mr{\B{A}})$. Here, $t$ is the clockwise renumbering of vertex: $i\mapsto i+1, n\mapsto 1$. Besides the usual simplicial relations,
it is clear that we have the cyclic relations \eqref{eq:cycrel1}--\eqref{eq:cycrel2}.

\end{example}

\begin{example} \label{Ex:BSn}
The underlying quiver with relations of the following QEC rooted at the quiver $A_3^1$ is the one considered in Example \ref{Ex:cube3}:
$$\vcenter{\xymatrix@R=0ex@C=9ex{
A_3^1 \ar@{_<-^>}[r]^{\B{S}_{2},\B{P}_{2}}_{\B{P}_2^r} & A_2  \ar@{_<-^>}[r]^{S_1,P_1}_{P_1^r} & A_1}
} $$

$$\cubetwo$$
The situation is similar for the quiver $A_3^{1*}$:
$$\vcenter{\xymatrix@R=0ex@C=9ex{
A_3^{1*} \ar@{_<-^>}[r]^{\B{S}_{[2,3]},\B{I}_{[2,3]}}_{\B{I}_{[2,3]}^r} & A_2  \ar@{_<-^>}[r]^{S_2,I_2}_{I_2^r} & A_1}
} $$
However, it cannot be realized as a QEC rooted at $A_3$ (exercise).

More generally, the QEC $\B{S}:$
$$\BSn{}{}{}$$
where $S_n$ is the $(n-1)$-subspace quiver, gives rise to the usual {\em n-cube category} $\Delta(\B{S})$.
Using Theorem \ref{T:rel}, we can read off the relations from the identities in Corollary \ref{C:sproj}, \ref{C:rel2} and Lemma \ref{L:pproj}.
In fact, we get the cubical relation \eqref{eq:cuberel1}--\eqref{eq:cuberel2} by substituting $\pi_i,\pi_{-i},\iota_{-i}$ by $S_i,P_i,P_i^r$ respectively.

If we consider the symmetry of $S_n$:
$$\BSn{\ar@(ul,ur)|{\mathbbm{t}_{n-1}}}{\ar@(ul,ur)|{\mathbbm{t}_{n-2}}}{\ar@(ul,ur)|{t_1}}$$
then we get the $n$-cube category with permutations.
\end{example}

\begin{example} \label{Ex:prism} Consider the QEC rooted at the quiver $A_4^1$: 
$$\vcenter
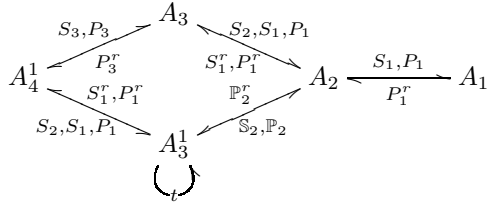
{\xymatrix@R=2ex@C=9ex{
&  A_3  \ar@{_<-^>}[dr]^-{S_2,S_1,P_1}_{S_1^r,P_1^r} \\
A_4^1	\ar@{_<-^>}[ur]^-{S_3,P_3}_{P_3^r} \ar@{^<-_>}[dr]_-{S_2,S_1,P_1}^{S_1^r,P_1^r} & & A_2 \ar@{_<-^>}[r]^-{S_1,P_1}_{P_1^r} & A_1\\
&  A_3^1  \ar@{^<-_>}[ur]_{\B{S}_2,\B{P}_2}^{\B{P}_2^r} \ar@(dl,dr)|{t}	}
} $$
Readers should be able to find all relations using Corollary \ref{C:sproj}, \ref{C:rel2}, Lemma \ref{L:pproj}, and Theorem \ref{T:rel}, and fill in the lower dimensional faces in the picture below.\\
This gives rises to a {\em triangular prism category}:
$\vcenter{\prism}$
\end{example}

\begin{example} \label{Ex:pyramid}
Consider the QEC rooted at the quiver $D_4$:
$$\vcenter
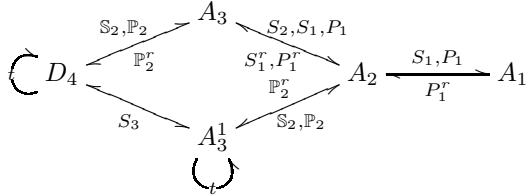
{\xymatrix@R=2ex@C=9ex{
&  A_3  \ar@{_<-^>}[dr]^{S_2,S_1,P_1}_{S_1^r,P_1^r} \\
D_4	\ar@{_<-^>}[ur]^{\B{S}_2,\B{P}_2}_{\B{P}_2^r} \ar@{^<-_>}[dr]_{S_3} \ar@(ld,lu)|{t} & & A_2 \ar@{_<-^>}[r]^-{S_1,P_1}_{P_1^r} & A_1	 \\
&   A_3^{1} \ar@{^<-_>}[ur]_{\B{S}_2,\B{P}_2}^{\B{P}_2^r}	\ar@(dl,dr)|{t}	}
} $$
Readers should be able to find all relations using Corollary \ref{C:sproj}, \ref{C:rel2}, Lemma \ref{L:pproj}, and Theorem \ref{T:rel}, and fill in the lower dimensional faces in the picture below. \\
This gives rise to a {\em square pyramid category}:
$\vcenter{\pyramid}$
\end{example}

It is also a good exercise to find the basic algebra of the above two QEC's. For the detail of all examples and more general results, we refer the readers to \cite{F0}. Although all above examples have nice geometric interpretation, many useful QEC's do not have {\em obvious} geometric meaning.

\begin{definition} \label{D:QECgen} Given any quiver $Q$ with a set of exceptional objects $\B{E}=\{E_1,E_2,\dots,E_n\}$, we can {\em generate} a QEC without reflections rooted at $Q$ {\em up to foldings} as follows.

First suppose that there is an automorphism $\sigma\in\Aut_0(Q)\times N$ such that $\sigma(E_i)=E_j$. Then {\em if we wish}, we can perform a {\em source-folding}, that is add the automorphism $\sigma$ on $Q$. Let $Q_i:=\pi_{E_i}(Q)$ and $\B{E}_i$ be the set of all exceptional objects of form $\varpi_{E_i}(E_j)$. Note that after performing the above source-folding, we can remove $E_i$ or $E_j$ from $\B{E}$ if $E_i\neq E_j$, but we still need both to compute $\B{E}_i$.

After all possible source-foldings, we have the following QEC:\\
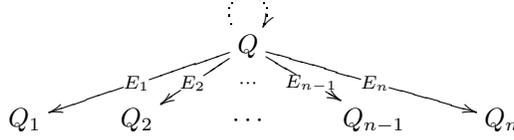
$$\QSn{Q}{Q_1}{Q_2}{Q_{n-1}}{Q_n}$$
Suppose that there is some equivalence $\sigma:\Mod(Q_i)\to\Mod(Q_j)$ such that $\sigma(\B{E}_{i})=\B{E}_{j}$. Then {\em if we wish}, we can perform a {\em target-folding}, that is
delete $Q_i$ and put the arrows $E_i$ toward $Q_j$. This $E_i$ should be understood as the original one composed with $\sigma$.
Recursively we can repeat this procedure for each $Q_i$ with $\B{E}_i$ until $\B{E}_i$ is empty.

In the above procedure, if we performed all possible foldings, then we call the generation {\em compressed}. If we took $\B{E}_i$ to be the set of all {\em connected} exceptional objects of form $\varpi_{E_i}(E_j)$, then we get a q-connected QEC. If we took $\B{E}_i$ to be only a part of connected exceptional objects, we call the generation {\em unsaturated}. By a slight abuse of language, by a QEC without foldings we mean that its underlying quiver has no multi-arrows between two vertices.

\end{definition}

Let us forget all sections for every example so far. We found that all examples fell into this type of construction, and all examples are compressed and q-connected. In Example \ref{Ex:pyramid}, we filtered $\varpi_{P_i}(S_i)=P_2$ out of the projections from $A_3$ to ensure q-connectedness.
Note that if $\B{E}$ is an exceptional collection, then all $E_{ij}:=\varpi_{E_i}(E_j)$ are exceptional by Theorem \ref{T:rel}. All examples so far are generated by exceptional collections. However, only in Example \ref{Ex:BAn}, $\B{E}$ is an exceptional sequence. Later we will see examples not generated by exceptional collections.

\begin{definition} A projection $\B{Q}(a)$ in a coface configuration $\B{Q}$ is called {\em augmented} if it is the only projection from $ta$ to $ha$ and $ha$ is a sink of $\B{Q}$.
$\B{Q}$ is called {\em GReedy} if its underlying quiver with relations (and sections) is GReedy.
\end{definition}

\begin{definition} \label{D:MGC} Given any GReedy coface configuration $\B{Q}$, a {\em GReedy completion} of $\B{Q}$ is a GReedy coface configuration obtained from $\B{Q}$ by adding some sections of unaugmented projections in $\B{Q}$.
\end{definition}

\begin{proposition} \label{P:MGC} Let $\B{Q}_1$ and $\B{Q}_2$ be two GReedy completions of a coface configuration $\B{Q}$ by adding sets of sections $S_1$ and $S_2$ respectively. Then the completion obtained by adding $S_1\cup S_2$ to $\B{Q}$ is still GReedy. In particular, there is  the unique maximal GReedy completion of $\B{Q}$.
\end{proposition}

\begin{proof} The paths not raising (resp. lowering) the degree correspond to the projections (resp. sections) in $\B{Q}$.
By the remark after Definition \ref{D:GReedy}, we verify the existence of the factorization for morphisms of form $f=ps$, where $p$ and $s$ is a projection and a section respectively. This is certainly true if we take the union of sections. Since projection (resp. section) is an epimorphism (resp. monomorphism), the uniqueness part of the GReedy condition follows from the uniqueness of the epi-mono factorization of morphisms in the category of sets.
\end{proof}

\begin{example} \label{Ex:aug} Consider the augmented $\B{A}_n:$
$$\BAnaug{}{A_}{\B{S}_n}{\B{S}_{n-1}^r}{S_2,S_1}{S_1^r}{S_1}{}{}.$$
The projection $S_1:A_1\to A_0$ is augmented, so we can not add $S_1^r$ to the completion. Otherwise the maximal GReedy completion will be the following QEC:
$$\BAnaug{}{A_}{\B{S}_n}{\B{S}_{n}^r}{S_2,S_1}{S_2^r,S_1^r}{S_1}{}{S_1^r}.$$ This is uninteresting because its algebra is semisimple.

The cube-shaped QEC in Example \ref{Ex:srooted} has three augmented projections to a common sink, but the QEC of Example \ref{Ex:BBn} has two augmented projections toward different sinks. For both examples, the maximal GReedy completions add no sections.
\end{example}

So far all QEC's in our examples are the maximal GReedy completions. This is the type of QEC's that we will consider throughout this notes. To complete a saturated QEC $\B{Q}$, the toy case that readers should keep in mind is the diamond diagram below Theorem \ref{T:rel}. If $E\perp F$, then we can add $\iota_E$ and $\iota_{E'}$. We can add $\iota_F$ if there is a relation of second kind $\pi_E\iota_F=\sigma\in\B{Q}_A$. More generally, we can add $\iota_E: Q_E\to Q$ if for any projection $\pi_F$ from $Q$, either $E\perp F$ and $\iota_{E'}\in\B{Q}_A$ or $\pi_F\iota_E=\sigma\in\B{Q}_A$.

\begin{example} Given any QEC $\B{Q}$ without foldings and reflections rooted at a quiver $Q_r$ and a representation $M$ of $Q_r$, we can associate a natural representation $\rho_M$ of $\B{Q}$ in $\Mod(Q_r)$ as follows. We first assign $M\xrightarrow{r_M}\tilde{\pi}_E(M)$
to $Q_r\xrightarrow{\pi_E}Q_E$, where $r_M$ is the universal morphism defined in Lemma \ref{L:OP}. Here, both $r_M$ and $\tilde{\pi}_E(M)$ may differ by an automorphism in $\Pic_k(Q_E)$.
Then we can recursively extend this assignment to the whole $\B{Q}$. We thus obtained a functor $\rho: \Mod(Q_r)\to\Rep(\B{Q},\Mod(Q_r)),\ M\mapsto \rho_M$.

Apply $\Hom_Q(M,-)$ to the universal morphism $M\xrightarrow{r_M}\tilde{\pi}_E(M)$, and we get $$\End_Q(M)=\Hom_Q(M,M)\xrightarrow{}\Hom_Q(M,\tilde{\pi}_E(M))=\End_Q(\tilde{\pi}_E(M)).$$
We thus obtain a representation of $\B{Q}$ in the category of $k$-algebras by applying $\Hom_Q(M,-)$ to the representation $\rho_M$.
\end{example}

\begin{example} \label{Ex:HQ} Let $H(Q)$ be the vector space spanned by isomorphism classes of representations of $Q$ over a field $k$. Given any QCC $\B{Q}$, we can associate a natural representation $H$ of $\B{Q}$ in $\Vect_k$ as follows. To any $Q_u\xrightarrow{F} Q_v$, we assign $H(Q_u)\xrightarrow{H(F)} H(Q_v)$, where $H(F)$ is the linear map induced by $M\mapsto F(M)$.
\end{example}

\section{Configuration with dimension vectors} \label{S:QECwd}

\begin{definition} A quiver coface configuration {\em with dimension vectors} is a QCC $\B{Q}$ with for each vertex quiver $\B{Q}(v)$ a dimension vector $\alpha_v$ such that all maps in $\B{Q}$ respect dimension vectors, that is \begin{enumerate}
\item For any loop $l$ on $v$, $\B{Q}(l)$ acts on $\Rep_{\alpha_v}(\B{Q}(v))$. \\
\item For each arrow $a: u\to v$ with $u<v$, the section $\B{Q}(a)$ embeds $\Rep_{\alpha_u}(\B{Q}(u))$ into $\Rep_{\alpha_v}(\B{Q}(v))$. Projections from $v$ to $u$ are the left adjoints of those sections.
\end{enumerate}
The corresponding notion for QEC is {\em QECwd}, quiver exceptional configuration {\em with dimension vectors}. Evidently the second condition for projections is equivalent to the following: \begin{enumerate}
\item[(2')] For each arrow $a:v\to u$ with $\B{Q}(a)=\pi_E$, we have that $E\perp\alpha_v$ and $\pi_E(\alpha_v)=\alpha_u$.
\end{enumerate}
Here $E\perp\alpha$ means that $E$ is left orthogonal to a {\em general representation} \cite{S2} $M$ of dimension $\alpha$, and $\pi_E(\alpha)$ is the dimension vector of $\pi_E(M)$.
\end{definition}

Given any quiver with dimension vector $(Q,\alpha)$ and a set of exceptional objects $\B{E}=\{E_1,E_2,\dots,E_n\}\perp\alpha$, we can {\em generate up to foldings} a QECwd rooted at $Q$ as before:
$$\QSn{(Q,\alpha)}{(Q_1,\alpha_1)}{(Q_2,\alpha_2)}{(Q_{n-1},\alpha_{n-1})}{(Q_n,\alpha_n)}$$
When doing the target-folding, the equivalence $\sigma$ we consider should also respect the dimension vectors.
Note that $\varpi_{E_i}(E_j)\perp \alpha_i$ is automatic, because $\tilde{\varpi}_{E_i}(E_j)\in\innerprod{E_i,E_j}\perp \alpha$.

\begin{definition} A QECwd $\B{Q}$ is called full (resp. fully q-connected) if for each vertex $v$, the projections $\pi_E$ range over all exceptional (resp. connected exceptional) $E\perp \alpha_v$.
\end{definition}

If we start with some $(Q,\alpha)$ and all exceptional $E\perp \alpha$, the QECwd generated is full. This is because for any $F\perp \alpha_i$, we have that $\iota_{E_i}(F)\perp\alpha$ and $\pi_{E_i}(\iota_{E_i}(F))=F$.


\begin{example} \label{Ex:BAndim}
Consider the fully q-connected QECwd rooted at $(A_{n+1},I_{n+1})$ compressedly generated by $\B{S}_{n}$:
$$\BAnpone{}{A_}{\B{S}_n}{\B{S}_{n-1}^r}{S_2,S_1}{S_1^r}{S_1}{}{}.$$
Note all exceptional $E\perp I_{n+1}$ can be presented as $0\to P_v\to P_u\to E\to 0$. However, only $\B{S}_n$ are connected.
Example \ref{Ex:BAn} can be realized as a colimit of above constructions.
\end{example}

\begin{example} \label{Ex:Dn1} Consider the fully q-connected QECwd $\B{D}_n^1$ rooted at $(D_{n+1},I_{n+1})$ compressedly generated by $\B{S}_{n}:$
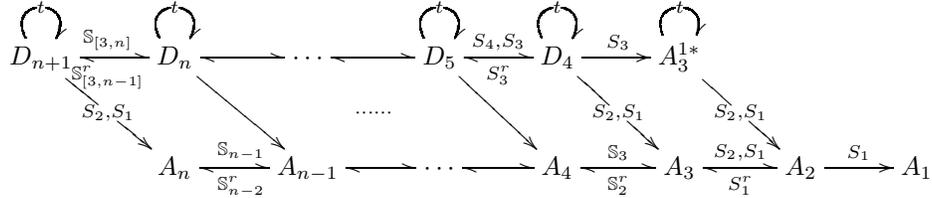
$$\BDnsimp{\ar@(ul,ur)|{t}}{S_2,S_1}$$

\begin{proposition} \label{P:Dn1} The algebra $k\B{D}_n^1$ is Morita equivalent to the path algebra
$$\Dncomp{a}{d}{c}{\ar@(ul,ur)|{t}}$$
with relations:
$$t^2=e_*,\quad tdt=d,\quad (c-ct)d=a(c-ct),\quad dd=0,\quad aa=0.$$
\end{proposition}
\begin{proof}[Sketch of Proof.] The proof is also similar to that of Proposition \ref{P:CDK}. Consider the breakings at the sequences $(\mathbf{S}_{n-2}',\mathbf{S}_{n-1},\mathbf{S}_{n-3}',\dots,\mathbf{S}_4,\mathbf{S}_2',\mathbf{S}_3,\mathbf{S}_1')$, where $\mathbf{S}_i \text{ (resp. $\mathbf{S}_i'$)} =(S_i,\dots,S_i,S_i)$ of the top (resp. bottom) row from the leftmost to the rightmost. We prove by induction using Theorem \ref{T:Br}.
\end{proof}
\end{example}

\begin{definition} A quiver $Q$ is called sink-rooted if every vertex is path-connected to the unique sink $\infty$.
\end{definition}


\begin{example} \label{Ex:srooted} This example generalizes the last two examples. Let $Q$ be any sink-rooted quiver, then for any $i\neq \infty$, $S_i\perp I_{\infty}$. It follows from Lemma \ref{L:sproj} that $\pi_i(Q)$ is also sink-rooted, and $\pi_i(I_\infty)=I_\infty$.
We consider the QECwd $\B{Q}$ generated by $S_i, i\neq\infty$ in \cite{F2}.
For another concrete example, let us consider the quiver $D_{(1,3,2)}$:
$$\srootedABC{a}{\mathbbm{b}_3}{c_1,c_2}.$$
We obtain a QEC without foldings:
$$\cube{}{}{}{}{}{}{}{}{}$$
where the quivers are \begin{align*}
& D_{(i,j)}: 1\xrightarrow{i\text{ arrows}} 2 \xleftarrow{j\text{ arrows}} 3, \\
& A_{(i,j)}: 1\xrightarrow{i\text{ arrows}} 2 \xrightarrow{j\text{ arrows}} 3, \\
& A_{(k)}: 1\xrightarrow{k\text{ arrows}} 2.\end{align*}
Its relations are commuting relations for all squares. If we like, we can perform the target-folding at $A_{(2)}$. Its maximal GReedy completion is clearly itself.

This class of examples can be thought of as a generalization of \cite{GLW} in four aspects. First, rooted trees is a subclass of sink-rooted quivers. Second, we have more projections and sections in our setting. For example, for a source $o$, the maps $S_o$ and $S_o^r$ are invisible in \cite{GLW}. Third, we allow foldings. Forth, dimension vectors play a role, which is crucial for our applications.
\end{example}

The next example seems more akin to type $D$ quivers than Example \ref{Ex:Dn1}.

\begin{example} We consider another fully q-connected QECwd $\B{D}_n^2$ rooted at $(D_{n+1}^*,\delta=(1,2,2\dots,2,1,1))$ generated by $(\B{S}_{[2,n-1]},I_{n+1},I_{n})$.
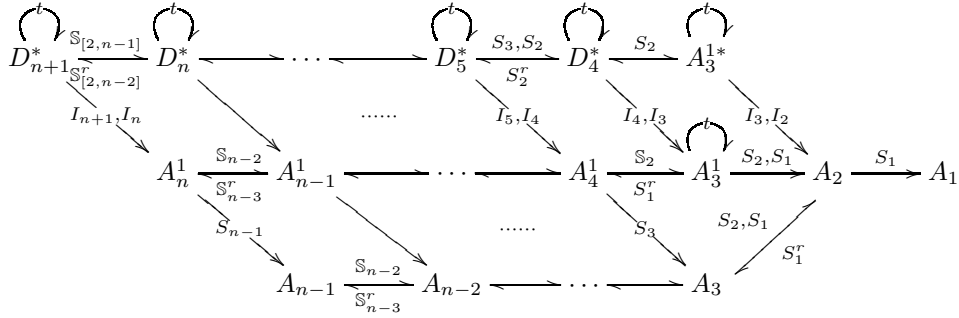
$$\BDntwo{I_{n+1},I_{n}}{I_5,I_4}{I_4,I_3}{I_3,I_2}{S_2,S_1}{\ar@(ul,ur)|{t}}$$
Note that this generation is not compressed because we did not identify $A_3^{1*}$ with $A_3^1$.
Besides the identities in Corollary \ref{C:sproj} and Lemma \ref{L:pproj}, one can verify that \begin{align*}
\pi_{I_i}(S_j)=S_{j-1}, & \text{  if } i\in\{n-1,n\}, j\in\{2,3,\dots,n-2\};\\
\pi_{I_i}(I_j)=S_{n-2}, & \text{  if } i,j\in\{n-1,n\}, i\neq j.
\end{align*}
Then the relations can be read off from Theorem \ref{T:rel}. The proof of the following proposition is similar to that of Proposition \ref{P:Dn1}.

\begin{proposition} \label{P:Dn2} The algebra $k\B{D}_n^2$ is Morita equivalent to the path algebra
$$\Dntwocomp{d}{c}{a'}{b}{a}{\ar@(ul,ur)|{t}}$$
with relations:
\begin{align*} t^2=e_*,\quad ta't=a',\quad cd=a'c,\quad ctd=a'ct,\quad ab=ba',\\
b(c-ct)=0,\quad dd=0,\quad a'a'=0, \quad aa=0.\end{align*}
The arrow $a'/b$ satisfies all relations of $a'$ and its composition with $t$ satisfies all relations of $b$. The arrow $a'/a/b$ satisfies all relations of $a',a'$ and $b$.
\end{proposition}
\end{example}

\begin{example} \label{Ex:BBn} Here is an example of a full QECwd $\B{B}^4$ rooted at a wild quiver with an imaginary Schur root generated by a non-exceptional collection. It is compressedly generated from $(B^4:=S_6,\alpha=(1,1,1,1,1,2))$ by $\{T_{ij}\}_{1\leqslant i<j\leqslant 5}$, where $\dim(T_{ij})=e_i+e_j+e_6$. Note that $T_{ij}\perp T_{kl}$ if and only if $\{i,j\}\cap\{k,l\}$ is non-empty. Moreover, if $\{i,j\}\cap\{k,l\}$ is empty, then $\pi_{ij}(T_{kl})=T_{\hat{m}}$ for some $m$, where $\dim T_{\hat{m}}=\mathbbm{1}-e_m$. Since $\innerprod{\mathbbm{1}-e_m,\mathbbm{1}-e_m}_{B^3}=0$, $T_{\hat{m}}$ is not exceptional.

$$\BBn{\ar@(dl,dr)|{\mathbbm{t}_4}}{\ar@(dl,dr)|{\mathbbm{t}_2,t}}{\ar@(dl,dr)|{t}}{\ar@(dl,dr)|{t}}{\ar@(ul,ur)|{t,t,t}}{\ar@(dl,dr)|{\mathbbm{t}_2}}
{T_{14},T_{24}}{\{T_{ij}\}_{i=1,2,3}^{j=4,5}}{\{T_{ij}\}_{i,j\neq 6}}$$
This QEC is related to the blow-up of points in the projective plane \cite{F1}. We recall the formulas of projections from \cite{F1}.
For $M_4\in \Rep_{\alpha}(B^4)$, we define
$$\pi_{ij}(M_4)=\extendedtameAfour{|jl|}{|il|}{|jm|}{|im|}{|jn|}{|in|}.$$
Here, we write $\pi_{ij}$ for $\pi_{T_{ij}}$, and use the notation $|il|$ to denote the determinant of linear map from $i$ and $l$ to $6$.
Moreover, $i<j, l<m<n, \{i,j,l,m,n\}=\{1,2,3,4,5\}$.
For $M_3\in \Rep_{\mathbbm{1}}(B^3)$, we define
$$\pi_{im}(M_3)=\extendedtameAthree{1}{2}{3}{4}{(jn)}{(jm)(in)}{(im)}{(kn)}{(km)(in)}.$$
Here, we use the notation $(ij)$ to denote the linear map on the arrow $i\to j$.
Moreover, $j<k,\{i,j,k\}=\{1,2,3\}$ and $\{m,n\}=\{4,5\}$.
For $M_2\in \Rep_{\mathbbm{1}}(B^2)$, we define
$$\pi_{i4}(M_2)=\extendedtameAtwo{1}{2}{3}{(i3)(j4)}{(j3)}{(i3)(34)}{(i4)},\qquad\text{and}\quad \pi_3(M_2)=\Bprime{1}{2}{3}{(13)(34)}{(14)}{(23)(34)}{(24)}.$$
For $M_1\in \Rep_{\mathbbm{1}}(B^1)$, we define
$$\pi_{13}(M_1)=\Bzero{1}{2}{(13)}{(12)(23)_u}{(12)(23)_l}.$$
The discrete automorphism group $S_5$ (resp. $S_3\times S_2,S_2,S_2,S_3,S_2\times S_2\times S_2$ ) acts on $B^4$ (resp. $B^3, B^2, B^1, B^0, B'$). We still denote the $i$-th transposition by $t_i$.
On each vertex, the automorphism group acts transitively on the projections, and satisfies obvious relations. Apart from obvious relations, we have relations up to automorphisms, which can be read off from:
\begin{align*} \text{On $B^4$}\qquad\qquad\\
\pi_{ki}(T_{jk})&=T_{j4}, & \pi_{ik}(T_{kj})&=T_{(j-2)5}, \\
\pi_{ki}(T_{kj})&=\begin{cases} T_{(j-2)4} & \text{ if } i<j \\ T_{(j-1)4} & \text{ if } i>j \end{cases}, &
\pi_{ik}(T_{jk})&=\begin{cases} T_{(j-1)5} & \text{ if } i<j \\ T_{j5} & \text{ if } i>j \end{cases}; \\
\text{On $B^3$}\qquad\qquad\\
\pi_{ki}(T_{kj})&=S_3, \qquad\quad\ \text{ if } i\neq j, \\
\pi_{ik}(T_{jk})&=\begin{cases} T_{(j-1)4} & \text{ if } i<j \\ T_{j4} & \text{ if } i>j \end{cases}; \\
\text{On $B^2$}\qquad\qquad\\
\pi_{i4}(T_{j4})&=T_{13} \qquad\quad\ \text{ if } i\neq j.
\end{align*}

One can check that the maximal GReedy completion is itself, so we do not have any section here.
\end{example}

\section{Delooping} \label{S:deloop}

In our coface configuration, loops are used to record information on symmetry. Those additional information can be effective to reduce the complexities of the homotopy theory. However, it causes trouble for us to take the classical homotopy or homology because usually those algebras are not basic and have infinite global dimensions. To simplify the computation, we need to get rid of those loops but hope to lose as little information as possible. So we propose to the following delooping, requiring the category $\D{C}$ to be a $k$-category.

\begin{definition} A {\em delooped coface configuration} $\B{Q}$ in a $k$-category $\D{C}$ is a quiver without loops with the following assignment. \begin{enumerate}
\item For each vertex $v$, we assign an object $\B{Q}(v)\in\D{C}$ (possibly with repetition).
\item For each arrow $a: u\to v$ with $u>v$ (resp. $u<v$), we assign a $k$-linear combination $\sum_i k_i\pi_i$ of projections (resp. $\sum_i k_i\iota_i$ of sections) in $\D{C}$.
\end{enumerate}
We require that all projections $\pi_i$ in the summation differ from each other only by automorphisms in $\D{C}$, that is, $\pi_i=\sigma_v\pi_j\sigma_u$ where $\sigma_u,\sigma_v$ are automorphisms of $\B{Q}(u),\B{Q}(v)$. We make the same requirement for sections in the linear combination.

As before, we also have the underlying quiver with relations $(Q,R)$ and the algebra $k\B{Q}$ associated to $\B{Q}$.
\end{definition}

\begin{definition} The $k$-linear category of category of quiver representations is the category $k\D{CC}(Q)$ with objects the $k$-vector spaces $H(Q)$ (Example \ref{Ex:HQ}) and morphisms $k$-linear spans of the continuous functors. A {\em delooped QCC} is a delooped coface configuration in the category $k\D{CC}(Q)$.
\end{definition}

\begin{example}  \label{Ex:symcle} The delooped QEC $\mr{\B{A}}^{\looparrowright}:$
$$\BAn{}{\mr{A}_}{\B{S}_{n+1}}{S^r,\B{S}_{n}^r}{S_2,S_1}{S^r,S_1^r},$$
where $S^r:\mr{A}_{k-1}\to \mr{A}_{k}$ equals to $\sum_{i=0}^{k-2} (-1)^{ik}tS_{k}^r t^i$, corresponds to our symcle category $\Delta(\mr{\B{A}}^{\looparrowright})$ considered in Section \ref{S:Qcomp}.
\end{example}

\begin{example} \label{Ex:Dn1dl} The algebra of this delooped QEC $\B{D}_n^{1\looparrowright}:$
$$\BDnsimp{}{S_2-S_1}$$
is Morita equivalent to the path algebra
$$\Dncomp{a}{d}{c}{}$$
with ``double complex" relation:
$$cd=ac,\quad dd=0,\quad aa=0.$$
Moreover, from the exact sequence $0\to S_1\to P_2\to S_2\to 0$, we get the long exact sequence of functors:
\begin{equation} \label{eq:longexact} \cdots\to\pi^{i-1}(-,S_2)\to\pi^i(-,S_1)\to\pi^i(-,P_2)\to\pi^i(-,S_2)\to\pi^{i+1}(-,S_1)\to\cdots.\end{equation}
Since we have the following resolutions: \begin{align*}
& 0\to S_1\to I_1\to I_2\oplus I_{2'}\to\cdots\to I_{n-1}\oplus I_{n-1'}\to I_n\to 0,  \\
& 0\to S_2\to I_2\to\cdots\to I_{n-1}\to I_n\to 0,  \\
& 0\to P_2\to I_1\to I_{2'}\to\cdots\to I_{n-1'}\to I_{n'}\to 0,
\end{align*}
we see that $\pi^i(-,S_2)$ (resp. $\pi^i(-,P_2)$) is the usual cohomology of the top (resp. bottom) row, and $\pi^i(-,S_1)$ is the usual total cohomology.
\end{example}

\begin{example} \label{Ex:Dn2dl} The algebra of this delooped QEC $\B{D}_n^{2\looparrowright}:$
$$\BDntwo{I_{n+1}-I_{n}}{I_5-I_4}{I_4-I_3}{I_3-I_2}{S_2-S_1}{}$$
is Morita equivalent to the path algebra
$$\Dntwocomp{d}{c}{a'}{b}{a}{}$$
with relations:
\begin{align*} & cd=a'c,\quad ab=ba',\\
& \ bc=0,\quad\ \ dd=0,\quad a'a'=0, \quad aa=0.\end{align*}
The arrow $a'/b$ satisfies all relations of $a'$ and $b$ except that its composition with $a':3'\to 2'$ is non-zero. The arrow $a'/a/b$ satisfies all relations of $a',a'$ and $b$.
\end{example}

\begin{example} Consider a delooped QEC $\B{B}^{4\looparrowright}:$
$$\BBndl{\sum_{i,j} (-1)^{i+j}T_{ij}}{\sum_{i,j} (-1)^{i+j}T_{ij}}{(\Id-t)(T_{14}-T_{24})}{(\Id-\sigma')S_3}{(\Id-\sigma)T_{13}}.$$
Here, $\sigma'=(t,t)$ is the automorphism of $B'$ and $t$ exchanges arrows in one pair; $\sigma$ is the automorphism of $B_0$ exchanging the upper and lower arrows.
The relations are that all paths of length two vanish. This algebra is already basic, and has finite representation type.
\end{example}

\section{A First Application} \label{S:App}

As we mentioned in the introduction, one get almost all classical (co)homology theories by varying $\D{C}$ in $S_{\D{C}}:S\to \D{C}$. Now we are going to vary $S$ among different $\Delta(\B{Q})$. We expect that those classical theories should have analogues, if not always possible.
Since we are mainly interested in noncommutative stuff, we will focus on generalizations of Hochschild and cyclic cohomology. Readers can easily adapt our construction to other settings. For example, to deal with commutative algebras, we can replace the tensor products below by the exterior products. In certain geometric setting with {\em nice} covers, we can consider $\Delta(\B{Q})$-nerves instead of simplicial nerves. Since our construction may involve addition (see examples below), in those cases the corresponding operation on covers is taking the union. So we need part of the Boolean algebra structure of sets to construct the coface maps. Recall that the coface maps in the simplicial nerves involve intersections only.

Let $Q$ be any quiver and $\alpha$ a dimension vector. By a tensor-decoration of $\Rep_\alpha(Q)$ by a $k$-module $V$, we mean the tensor product space labeled by the arrows of $Q$: $$V^{\otimes(Q,\alpha)}:=\bigotimes_{a\in Q_A} \Hom(k^{\alpha(ta)},k^{\alpha(ha)})\otimes V.$$
It is useful to visualize those ``pure" tensors on the representation space.

We recall the construction of Hochschild cohomology. Let $A$ be a $k$-algebra. Consider the quiver $A_{n+1}$ with dimension vector $(1,1,\dots,1)$ decorated by $A$:
$$\An{1}{a_1}{2}{n}{a_{n}}{n+1}{a_2}{a_{n-1}}.$$
The above picture represents a pure tensor $a_1\otimes a_2\otimes\cdots\otimes a_n$.
The Hochschild complex is related to the QEC with this dimension vector. Let us consider the coface maps given by
\begin{align*}
& \partial_1(a_1\otimes a_2\otimes\cdots\otimes a_n):=a_2\otimes a_3\otimes\cdots\otimes a_n,\\
& \partial_i(a_1\otimes a_2\otimes\cdots\otimes a_n):=a_1\otimes\dots\otimes a_{i} a_{i+1}\otimes\dots\otimes a_{n}\quad \text{ if $i>0$},\\
\intertext{and the face maps (degeneracies) given by}
& \sigma_i(a_1\otimes a_2\otimes\cdots\otimes a_n):=a_1\otimes\dots\otimes a_{i}\otimes 1\otimes a_{i+1}\otimes\dots\otimes a_{n}.
\end{align*}
In our language, if we set $\rho(A_{n},I_{n})=A^{\otimes(A_n,I_n)},\rho(S_i)=\partial_i,$ and $\rho(S_i^r)=\sigma_i$,
then $\rho$ is a representation of $\B{A}$ in the category of $k$-algebras. Its quiver complex is the {\em (reduced) Hochschild complex}.

For any $A$-$A$-bimodule $M$, we apply the functor $M_{A\otimes A^{\op}}\otimes -$ to the Hochschild complex, and we get the Hochschild homology complex of $A$ with coefficients in $M$. Similarly, we can apply the functor $\Hom_{A\otimes A^{\op}}(M,-)$, and obtain the Hochschild cohomology complex of $A$ with coefficients in $M$.

%

In what follows, for any QECwd $\B{Q}$, we always assign to the vertex $(Q,\alpha)$ the space $A^{\otimes(Q,\alpha)}$. To a projection $S_i$, we always assign the matrix multiplication corresponding to the contraction in Corollary \ref{C:sproj}. To a section $S_i^r$, we also assign the insertion of matrix $E_j$'s as in Corollary \ref{C:sproj}.
\begin{example} \label{Ex:NAsrooted} The above assignment clearly defines representations in $k\-\alg$ of QECwd's of Example \ref{Ex:srooted}. This class of examples are also studied in detail in \cite{F2}. In particular, we get representations of $\B{D}_n^1$ and $\B{D}_n^{1\looparrowright}$.
\end{example}

However, it is not true that given any QECwd, we can always attach to it a cohomology theory for noncommutative algebras. This is something expected, because our QEC's are constructed from the representation theory of $Q$, where we work with a nice field and have a group action. For example, we only know a $\B{D}_n^2$-cohomology theory for commutative algebras $k$-algebras.

\begin{example} \label{Ex:NADn2} To construct such a theory, we remain to define the coface maps corresponding to $I_k$.
Let $$M=\Dndual{r}{M_2}{M_{n-2}}{c_1}{c_2}{n-1}{\quad n}{n+1},$$
where $r,c_i,$ and $M_j$ are $1\times 2,2\times 1,$ and $2\times 2$ matrix with entries in $A$ respectively.
We define
\begin{align*}
\pi_{I_k}(M)&=\Anonei{|M_2|}{|M_i|}{|N|}{l_k}{}{n}{n-1},
\end{align*}
where $l_k=rM_2\cdots M_{n-2}c_{k}$ and $N$ is the matrix $(c_1,c_2)$. The symbol $|M|$ is the determinant of $M$.
To appreciate this new theory, we strongly recommend readers to work out this verification.

\end{example}

\begin{example} Let $\B{B}^3$ be the full {\em subQEC} of $\B{B}^4$ rooted at $B^3$. For $\B{B}^3$, if we define its representation by mimicking Example \ref{Ex:BBn}, we only get a cohomology theory for commutative algebras. For $\B{B}^4$, we even cannot find such a theory for commutative algebras.
\end{example}

The quiver complexes can deal with not only the traditional bimodules but also $n$-ary modules over $n$-ary algebras. More generally, they provide a natural framework to treat operads. This is one of the topics for our subsequential work. We must point out that this approach is different from so-called operadic (co)homology in \cite[Chapter 6,12]{LV}, which is simplicial in nature.

\section{Discussion} \label{S:Discuss}
We have several important problems, which are very difficult for us at this stage.
We saw that for all QEC's in this notes, we can find their basic algebras by breakings at a sequence, in other words their algebras are all (weakly) crisp. However, we do not know how general it is. We also saw that in general the choice of breakings are not unique, and to find the breakings is a very hard task.
\begin{problem} Characterize those QRS's whose basic algebras can be obtained by consecutive breakings. For those QRS's, design a good algorithm to find most economic choices of breakings.

\end{problem}

We can easily write the fundamental relations using the machinery developed in Section \ref{S:OP}. However, it is extremely hard to verify whether those fundamental relations generate all relations. This is true for all our examples, but we avoid writing down the checking, because the procedure is too painful. So the next question is
\begin{problem} Determine when the two kinds of fundamental relations generate all relations. Can we find an example involving other kinds of relations?
\end{problem}
All the fundamental relations of second kind that we found so far are of the form in Corollary \ref{C:rel2}, but we wonder
\begin{problem} Are there any fundamental relations of second kind $\tilde{\pi}_E\iota_F=\sigma$ for other choices of $E$ and $F$?
\end{problem}

We followed some topologists and introduced the GReedy condition. This condition turns out to be very useful in the model category theory \cite{BM}, but seems not so natural to us. It is suggested by our intuition that adding sections while keeping this condition can reduce the complexity of the representation theory. However, we have examples where this intuition fails.
Below Example \ref{Ex:aug}, we proposed some preliminary rules to ``complete" a QEC. However, we cannot guarantee in general that this will give us the maximal GReedy one.

\begin{problem} Design a good algorithm to find the maximal GReedy completion of any QEC. Is there any other approaches to produce interesting and meaningful QEC's?
\end{problem}

In this notes, one of our main examples is the ``Q-homotopy theory of type $D$".
\begin{problem} Is there a cyclic analog of this type $D$ theory? More generally, what is a cyclic analog for any Q-homotopy theory?
\end{problem}

\appendix

\section{Unnormalized Constructions and Q-Homology}

We keep the setting in Section \ref{S:QRS}.
\begin{definition}
An {\em unnormalized quiver complex (resp. cocomplex) functor} $\mu:\Mod(A)\to\Mod(A^b)$ is a $k$-linear exact functor such that $\nu$ is a quotient functor (resp. subfunctor) of $\mu$. It is called {\em classical} if it is induced from an algebra morphism $A^b\to A$ such that $e_{i1}\mapsto e_i$. It is called {\em splitting} if it is a quiver cocomplex functor in the meantime.
\end{definition}

Classical unnormalized quiver complex functors always exist, for example, the algebra inclusion defined by $e_{i1}\mapsto e_i$ and identity elsewhere. In fact, it is splitting, but this trivial case is uninteresting to us. We will deal exclusively with quiver complexes, and hope that readers can formulate the corresponding statement for quiver cocomplexes.

Let $\lambda$ be the subfunctor of $\mu$ such that the normalized functor $\nu$ is the quotient $\mu/\lambda$. Since both $\mu$ and $\nu$ are exact, the snake lemma implies $\lambda$ is exact as well.
We denote by $\D{D}$ the set of all modules of form $\lambda(M)$.
\begin{definition}
Fixing a quiver complex functor $\mu$, the composite $\underline{H}=q \mu$ is called the {\em Q-homology functor} of $A$. For any $T\in\Mod(A^b)$, the {\em $i$-th classical homology relative to $T$, or $T$-homology} is the functor $H_i(T,-):=\Ext_{A^b}^i(T,\mu(-))$.

We call the pair $(\mu,\nu)$ (resp. the triple $(\mu,\nu,T)$) a {\em homotheory} (resp. {\em $T$-homotheory}).
It is called {\em consistent} if $\underline{\pi}$ and $\underline{H}$ (resp. $\pi_i(T,-)$ and $H_i(T,-)$) are the same.
\end{definition}

\begin{remark} Unlike the homotopy functors, the homology functors depend on the choice of $\mu$.
Clearly, being consistent is equivalent to that $\Ext_{A^b}^i(T,-)$ vanishing on $\D{D}$, or equivalently on $\lambda(\nu^{-1}(S_v))$ for each simple module $S_v$ in $\Mod(A^b)$.

For cohomology theory, one can work with quiver cocomplexes, and define the $T$-cohomology functors by $\Ext_{A^b}^i(\mu(-),T)$.
\end{remark}

Now we generalize to representations in a $k$-linear abelian category $\D{A}$. For simplicity, we keep our assumption in Section \ref{S:Qcomp} that $\D{A}$ is the category $\Mod(B)$ for some $k$-algebra $B$. We are going to construct a functor $\mu_{\D{A}}:\Mod(A\otimes B^{\op})\to\Mod(A^b\otimes B^{\op})$ from $\mu$. Given any $M\in \Mod(A\otimes B^{\op})$, we can give $\mu(M)$ a right $B$-module structure. We view the right multiplication by $f\in B$ as an $A$-module homomorphism, and define the action of $f$ on $\mu(M)$ by $\mu(f)$. This action is compatible with the $A^b$-module structure. By our construction, we have the commutative diagram:
$$\vcenter{\xymatrix{
\Mod(A\otimes B^{\op}) \ar[r]^{\mu_{\D{A}}}_{\nu_{\D{A}}} \ar[d]_F & \Mod(A^b\otimes B^{\op}) \ar[d]^F \\
\Mod(A) \ar[r]^{\mu}_{\nu} & \Mod(A^b) }
}$$
where $F$ is the forgetful functor $\Mod(B)\to\Mod(k)$. We conclude that $\nu_{\D{A}}$ is a quotient functor of $\mu_{\D{A}}$ as well.
A sneaky way to define $\mu_{\D{A}}$ is to say that it is the tensor product of $\mu$ and the identity functor on $\Mod(B)$.
When $\mu$ is classical, $\mu_{\D{A}}$ is the functor induced from the algebra map $A^b\otimes B \to A\otimes B$.


\begin{definition}
The composite $\underline{H}_{\D{A}}=q \mu_{\D{A}}$ is called the {\em Q-homology functor in $\D{A}$}. For any $T\in\Mod(A^b)$, the {\em $i$-th classical homology relative to $T$}, or {\em $T$-homology in $\D{A}$} is the functor $H_i(T,-):=\Ext_{A^b}^i(T,\mu_{\D{A}}(-))$.

The homotheory and being consistent are defined completely analogously. It is quite obvious that if $(\mu,\nu)$ or $(\mu,\nu,T)$ is consistent, then $(\mu_{\D{A}},\nu_{\D{A}})$ or $(\mu_{\D{A}},\nu_{\D{A}},T)$ is consistent as well.
\end{definition}

\begin{example} In Proposition \ref{P:CDK},
we choose a classical $\mu$ as the one induced by $d\mapsto \sum_{i=1}^{k} (-1)^i\pi_i$ for $d:k\to k-1$. It follows from the simplicial relations that $dd=0$, so this is indeed an algebra inclusion. This $\mu$ is the usual (truncated) Moore complex functor in the simplicial theory. It is splitting.
Using the formula of $P$, it is not hard to verify that $\lambda(\nu^{-1}(S_i))$ are all injective, so $(\mu,\nu)$ is consistent. This is well-known \cite[Theorem 8.3.8]{We}.
\end{example}

\begin{example} In Proposition \ref{P:SDK},
we choose a classical $\mu$ as the one induced by $B\mapsto (1+(-1)^kt)\iota^r$ for $B:k-1\to k$ and $d$ as before. One can check that this is indeed an algebra inclusion \cite[9.8]{We}. This $\mu$ is not splitting. Although most $\lambda(\nu^{-1}(S_i))$ are not injective, it follows from the SBI sequence and the consistency of the simplicial theory that $(\mu,\nu,S_1)$ is consistent up to degree $n-1$, that is, $\pi_i(S_1,-)=H_i(S_1,-)$ for $i<n-1$. For the untruncated version, it is well-known that the $S_1$-homotheory is consistent \cite[9.8.4]{We}.
\end{example}

\begin{example} In Proposition \ref{P:CuDK}, for cubical objects
we choose a classical $\mu$ as the one induced by $d_i\mapsto \pi_i+\pi_{-i}$. It follows from the cubical relations that $d_{j+1}d_i=d_jd_i$, so this is indeed an algebra inclusion. Similarly for cubical objects with permutations, we set $d\mapsto \pi_1+\pi_{-1}$ and $t_i\mapsto t_i$.
We leave it for readers to verify that the direct summands of $\lambda(\nu^{-1}(S_1))$ contain all simples except $S_1$.
So there is no non-trivial consistent homotheory available. Even so, the functor $\mu$ seems still interesting.
\end{example}

\begin{example} In Example \ref{Ex:Dn1dl}, we choose a classical $\mu$ as the one induced by $d\mapsto \sum_{i=3}^{k} (-1)^{i+1}S_i, a\mapsto \sum_{i=1}^{k} (-1)^iS_i, c\mapsto S_2-S_1$, where $d:D_{k+1}\to D_k$ and $a:A_{k+1}\to A_k$. It is easy to check that this defines an algebra inclusion. Although most $\lambda(\nu^{-1}(S_i))$ are not injective, it follows from the long exact sequence \eqref{eq:longexact} and the consistency of the simplicial theory that $(\mu,\nu,S_1)$ is consistent. We can define a similar quiver complex functor for $\B{D}_n^{2\looparrowright}$ of Example \ref{Ex:Dn2dl}.
\end{example}

\begin{example} We consider some natural choices of $\mu$ for Example \ref{Ex:srooted} in \cite{F2}. Similar results were also obtained in \cite{GLW}. However, for sink-rooted quivers rather than rooted trees, the ``generic" examples look like the example generated from the quiver $D_{(1,3,2)}$, where $k\B{Q}$ is already basic, and $\mu$ has to be the identity.
\end{example}

\section*{Acknowledgement}
The author wants to thank Philip Hackney, Birge Huisgen-Zimmermann, Ryan Kinser, and Igor Kriz for helpful comments.
He would like to thank Philip Hackney and Professor Bernhard Keller for drawing his attention to the paper \cite{GLW} and \cite{Ci} respectively;
Ryan Kinser for pointing out an error in Section \ref{S:Qcomp}.
The author is grateful to Professor Igor Kriz for encouraging him to apply graduate school back in 2003. He always feels indebt to his advisor Harm Derksen for too many things that cannot be written in a single page.

\bibliographystyle{amsplain}

\end{document}